\documentclass[11pt, a4paper]{article}

\usepackage{bm}
\usepackage{amssymb, amsmath, amsthm}
\usepackage{color}
\usepackage{array}
\newcolumntype{C}[1]{>{\centering}m{#1}}
\usepackage{enumitem}
\usepackage{graphicx}
\usepackage{subcaption}
\usepackage{tikz}
\usepackage{authblk}

\newcommand{\grobner}{Gr\"{o}bner }
\newcommand{\field}[1]{\mathbb{#1}}
\newcommand{\pset}[1]{\mathcal{#1}}
\newcommand{\p}[1]{\bm{#1}}
\newcommand{\kx}{\field{K}[\p{x}]}
\newcommand{\fk}{\field{K}}

\newcommand{\ideal}[1]{\mathfrak{#1}}
\newcommand{\bases}[1]{\langle #1 \rangle}

\DeclareMathOperator{\rank}{rank}
\DeclareMathOperator{\rterm}{\mathfrak{R}}
\DeclareMathOperator{\ess}{Ess}

\DeclareMathOperator{\lc}{lc}
\DeclareMathOperator{\lt}{lt}
\DeclareMathOperator{\lv}{lv}

\DeclareMathOperator{\term}{\mathfrak{T}}
\DeclareMathOperator{\ini}{ini}
\DeclareMathOperator{\sat}{sat}

\DeclareMathOperator{\lcm}{lcm}

\newtheorem{theorem}{Theorem}[section]
\newtheorem{proposition}[theorem]{Proposition}
\newtheorem{lemma}[theorem]{Lemma}
\newtheorem{corollary}[theorem]{Corollary}

\theoremstyle{definition}
\newtheorem{definition}[theorem]{Definition}
\newtheorem{example}[theorem]{Example}

\theoremstyle{remark}

\numberwithin{equation}{section}

\begin{document}

\title{On the Reduced Gr\"obner Bases of Blockwise Determinantal Ideals}

\author{Chenqi Mou}

\author{Qiuye Song}
\affil{LMIB --  School of Mathematical Sciences \\
Beihang University, Beijing 100191, China\\
\{chenqi.mou/qiuye.song\}@buaa.edu.cn}


\date{}
\maketitle

\begin{abstract}
  Blockwise determinantal ideals are those generated by the union of all the minors of specified sizes in certain blocks of a generic matrix, and they are the natural generalization of many existing determinantal ideals like the Schubert and ladder ones. In this paper we establish several criteria to verify whether the \grobner bases of blockwise determinantal ideals with respect to (anti-)diagonal term orders are minimal or reduced. In particular, for Schubert determinantal ideals, while all the elusive minors form the reduced \grobner bases when the defining permutations are vexillary, in the non-vexillary case we derive an explicit formula for computing the reduced \grobner basis from elusive minors which avoids all algebraic operations. The fundamental properties of being normal and strong for W-characteristic sets and characteristic pairs, which are heavily connected to the reduced \grobner bases, of Schubert determinantal ideals are also proven. 
\end{abstract}



\section{Introduction}
\label{sec:intro}

\grobner bases of multivariate polynomial ideals, first introduced by Buchberger in his PhD thesis \cite{B1965A}, are a finite set of generators which characterize the initial ideals. Good properties of \grobner bases together with effective algorithms for computing them \cite{B85G,F1999A,F2002A} make \grobner bases an indispensable tool to study multivariate polynomial ideals in an algorithmic way \cite{CLO1997I}.

Determinantal ideals are those generated by specific minors of a generic matrix with variables as its entries and they are a fundamental algebraic object of general interest with deep connections to many subjects like representation theory and combinatorics \cite{BV06D}. In the study of \grobner bases of determinantal ideals, Sturmfels first introduces the Robinson-Schensted-Knuth correspondence \cite{Knu70p} to use Young bitableaux \cite{FUL97Y} as the combinatorial representations of standard monomials, namely products of comparable minors in the theory of straightening law \cite{DRS74o}. The theories and methods are considerably developed for the study of \grobner bases of determinantal ideals in \cite{BC03g,BCR22d} and the \grobner bases of different kinds of determinantal ideals like the ladder ones are also identified \cite{CON95L,GOR07M}. 

Recently there is a rising trend to study the \grobner bases of polynomial ideals with combinatorial significance. In particular, the Schubert determinantal ideal $I_w$ of a permutation $w\in S_n$ is the corresponding ideal of matrix Schubert variety and is closely related to the (double) Schubert polynomial of that permutation, a central tool in Schubert enumerative calculus \cite{LS82p, BJS93s,FS94s,dua05m}. Fulton simplified the original generating set of $I_w$ by introducing the essential set of $w$ \cite{FUL92F}. These new generators, called Fulton generators in the literature, are further identified as the \grobner basis of $I_w$ with respect to (w.r.t. hereafter) any anti-diagonal term order \cite{KM05g}. In particular, deep relationships among the initial ideal of $I_w$, its Stanley-Reisner simplicial complex, the double Schubert polynomial of $w$, and reduced pipe dreams of $w$ are established too \cite{KM05g}, fully revealing the potential of identifying \grobner bases of the Schubert determinantal ideals, for they furnish an effective tool to study such ideals in both theoretical and computational ways. Then these results were extended to the diagonal term order and different generators of $I_w$ for different kinds of permutations and to other kinds of ideals: Fulton generators are identified as the \grobner basis of $I_w$ w.r.t. any diagonal term order for a vexillary (namely 2143-avoiding) permutation $w$ \cite{KMY09g}; CDG generators are identified as the \grobner basis of $I_w$ w.r.t. any diagonal term order for a banner permutation $w$, together with established relationships with bumpless pipe dreams \cite{HPW22G,LLS21b}; the \grobner bases of the Kazhdan-Lustig ideals are also obtained \cite{WY12a,EH19a}.

In this paper we are interested in the reduced \grobner bases of determinantal ideals like the Schubert and ladder ones. A \grobner basis of a polynomial ideal is \emph{minimal} if its polynomials do not permit redundant leading terms and is further called \emph{reduced} if no reduction of the terms of its polynomials is possible. Reduced \grobner bases of polynomial ideals are unique w.r.t. the term orders and thus they are canonical representations of the ideals. Furthermore, reduced \grobner bases w.r.t. the lexicographic term order have deep connections to triangular sets, another kind of fundamental tool in computational ideal theory \cite{l91n,W2016o,WDM20d}. Reduced \grobner bases can be obtained from ordinary ones by first transforming them into minimal ones and then performing reduction among the polynomials. With Fulton generators identified as \grobner bases of Schubert determinantal ideals, minimal \grobner bases for such ideals were studied in \cite{GY22M}, where the authors introduced the notion of elusive minors and proved that they form minimal \grobner bases of Schubert determinantal ideals. 

In this paper we study reduced \grobner bases of the so-called ``blockwise'' determinantal ideals which are inspired by the definitions of Schubert, ladder, and mixed ladder determinantal ideals, for their generators are all in the form of taking the union of all the minors of the same sizes from a sequence of specified blocks of the generic matrix. To study reduced \grobner bases we focus on reducibility of terms of polynomials in the \grobner bases, which are themselves homogeneous polynomials expanded from the minor generators. The main contributions of this paper are: (1) we unify many existing determinantal ideals in the literature in the framework of blockwise determinantal ideals and present several sufficient conditions for their minor generators to form the \grobner bases; (2) by introducing the notion of length of an arbitrary term in a block we establish the criterion for testing reducibility of this term by the minor generators in that block of the blockwise determinantal ideal, and it naturally leads to new easy-to-use criteria to verify minimal and reduced \grobner bases; (3) in particular, we are able to prove that elusive minors of Schubert determinantal ideals, which are minimal by \cite{GY22M}, form their reduced \grobner bases if and only if the permutations are vexillary (when we were preparing this manuscript we found that the same results also appear in \cite{S23m}); (4) for non-vexillary permutations we fully characterize the process of polynomial reduction among the elusive minors to provide an explicit formula for performing the reduction which avoids all algebraic polynomial operations; (5) based on the above theories for reduced \grobner bases of Schubert determinantal ideals, two fundamental properties of W-characteristic sets and the characteristic pairs of such ideals are proven: this may initialize the study of Schubert determinantal ideals via the tool of triangular sets.

\section{Preliminaries}
\label{sec:pre}

Throughout this paper $\fk$ is a field. For any two positive integers $m \leq n$, we use $[n]$ and $[m, n]$ to denote the sets $\{1, \ldots, n\}$ and $\{m, \ldots, n \}$ respectively. 

\subsection{\grobner bases and W-characteristic set}
\label{sec:gb-ts}

\subsubsection{\grobner bases}
\label{sec:gb}

Consider the multivariate polynomial ring $\fk[x_1, \ldots, x_n]$ in $n$ variables. A \emph{term} in $\fk[x_1, \ldots, x_n]$ is a product of pure powers of the variables $x_1, \ldots, x_n$. For the set of all terms in $\fk[x_1, \ldots, x_n]$, one can associate a \emph{term order} $<$ satisfying the following conditions: (1) it is a total order; (2) for any three terms $\p{u}, \p{v}$ and $\p{w}$, $\p{u} < \p{v}$ implies $\p{u}\p{w} < \p{v}\p{w}$; (3) any subset of terms contains a minimal element w.r.t. $<$. Term orders play a fundamental role in the theory of \grobner bases, and typical term orders include the lexicographic and degree reversed lexicographic ones, etc. \cite{CLO1997I}. In particular, fix a variable order $x_{i_1} < \cdots < x_{i_n}$. Then the lexicographic term order $<_{\rm lex}$ induced by it is defined as follows: $x_1^{\alpha_1}\cdots x_n^{\alpha_n} <_{\rm lex} x_1^{\beta_1}\cdots x_n^{\beta_n}$ if and only if there exists an integer $k \in [n]$ such that $\alpha_{i_j} = \beta_{i_j}$ for $j = k+1, \ldots, n$ and $\alpha_{i_k} < \beta_{i_k}$. 

The greatest term of a polynomial $F\in \fk[x_1, \ldots, x_n]$ w.r.t. a term order $<$ is called its \emph{leading term} and denoted by $\lt_<(F)$. The coefficient of $F$ w.r.t. $\lt_<(F)$ is called its \emph{leading coefficient} and denoted by $\lc_<(F)$. When no ambiguity may occur, $\lt_{<}(F)$ and $\lc_<(F)$ are also written as $\lt(F)$ and $\lc(F)$ respectively for simplicity. Fix a term order $<$, then for any subset $\pset{F} \subseteq \fk[x_1, \ldots, x_n]$, denote $\lt(\pset{F}) = \{\lt(F): F \in \pset{F}\}$.

\begin{definition}\rm \label{def:gb}
Let $\ideal{I}$ be an ideal in $\fk[x_1, \ldots, x_n]$ and $<$ be a term order. Then a finite polynomial set $\pset{G} \subseteq \ideal{I}$ is called a \emph{\grobner basis} of $\ideal{I}$ w.r.t. $<$ if $\bases{\lt(\pset{G})} = \bases{\lt(\ideal{I})}$.   
\end{definition}

\grobner bases possess many desirable properties and are a fundamental tool to study multivariate polynomial ideals algorithmically \cite{B1965A,B85G}. A \grobner basis $\pset{G}$ w.r.t. a term order $<$ is said to be \emph{minimal} if for each $G\in \pset{G}$, $\lt(G)$ cannot be divided by the leading term of any other polynomial in $\pset{G}$ and $\lc(G)=\pm 1$. In other words, minimal \grobner basis does not contain polynomials with redundant leading terms. Note that usually it is required that polynomials in the minimal \grobner basis are \emph{monic}. In this paper we work with minors of a generic matrix as the generators of determinantal ideals and their leading coefficients are always $\pm 1$, and loosening the requirement on the leading coefficients a bit saves us from the trouble of changing the signs of the leading coefficients. 

Denote the set of all terms effectively appearing in a polynomial $F \in \fk[x_1, \ldots, x_n]$ by $\term(F)$. Let $P, Q \in \fk[x_1, \ldots, x_n]$ be two polynomials and $<$ be a term order. Then $P$ is said to be \emph{reducible} modulo $Q$ if there exists a term $\p{t} \in \term(P)$ such that $\lt(Q) | \p{t}$, otherwise $P$ is \emph{reduced} modulo $Q$. If $P$ is reducible, then we can reduce $P$ modulo $Q$ by $\p{t}$ to have $\tilde{P} := P - \frac{\lc(P)}{\lc(Q)} \frac{\p{t}}{\lt(Q)}Q$. This process is denoted by $ P\xrightarrow[\p{t}]{Q} \tilde{P}$. Again if $\tilde{P}$ is reducible modulo $Q$, we can further reduce $\tilde{P}$. We can repeat this kind of reduction until the result $\overline{P}$ is reduced, and this process is denoted by $ P \xrightarrow[*]{Q} \overline{P}$. Furthermore, let $\pset{Q} \subseteq \kx$ be a polynomial set. Similarly, $P$ is said to be \emph{reducible} modulo $\pset{Q}$ if $P$ is reducible modulo some $Q \in \pset{Q}$, otherwise $P$ is \emph{reduced}. We can reduce $P$ modulo $\pset{Q}$ repeatedly as long as the result is reducible until we have a reduced polynomial $\overline{P}$, and this process is denoted by $ P \xrightarrow[*]{\pset{Q}} \overline{P}$.

Reduction modulo a polynomial (set) is the most important operation in the construction of \grobner bases and in the computational manipulation of multivariate polynomial ideals like performing the ideal membership test with \grobner bases. In particular, from a minimal \grobner basis $\pset{G}$, one can replace each $G \in \pset{G}$ by $\overline{G}$ with the reduction $ G \xrightarrow[*]{\pset{G}\setminus \{G\}} \overline{G}$ to have the reduced \grobner basis defined below \cite[Chapter~2.7]{CLO1997I}.

\begin{definition}\label{def:reducedGB}
  Let $\pset{G}$ be a \grobner basis of an ideal $\ideal{I} \subseteq \fk[x_1, \ldots, x_n]$  w.r.t. a term order $<$. Then $\pset{G}$ is said to be \emph{reduced} if for each $G \in \pset{G}$, $G$ is reduced modulo $\pset{G} \setminus \{G\}$ and $\lc(G) = \pm 1$.
\end{definition}

By definition, it is clear that a reduced \grobner basis is also minimal. For an arbitrary polynomial ideal $\ideal{I} \subseteq  \fk[x_1, \ldots, x_n]$, its reduced \grobner basis w.r.t. a fixed term order is unique (in our slightly weaker definition, up to the signs of the leading coefficients). Reduced \grobner bases, due to their uniqueness, are the standard representation of polynomial ideals and they are heavily related to the W-characteristic sets described in the following.

\subsubsection{Triangular set and W-characteristic set}
\label{sec:ts}

Fix an order $x_1 < \cdots < x_n$ on all the variables of the polynomial ring $\fk[x_1, \ldots, x_n]$. For a polynomial $F \in \kx$, the greatest variable effectively appearing in $F$ is called its \emph{leading variable} and denoted by $\lv(F)$. Writing $F$ as a univariate polynomial in $\lv(F)$, its leading coefficient w.r.t. $\lv(F)$ is called the \emph{initial} of $F$ and denoted by $\ini(F)$.

\begin{definition}\rm\label{def:ts}
An ordered polynomial set $\pset{T} = [T_1, \ldots, T_r] \subseteq \fk[x_1, \ldots, x_n]$ is called a \emph{triangular set} if $\lv(T_1) < \cdots < \lv(T_r)$. Furthermore, $\pset{T}$ is said to be \emph{normal}  if for each $i \in [r]$, $\ini(T_i)$ does not involve any of $\lv(T_1), \ldots, \lv(T_r)$. 
\end{definition}

For a triangular set $\pset{T}$, its \emph{saturated ideal} $\sat(\pset{T})$ is defined to be the saturation of $\bases{\pset{T}}$ w.r.t. $I$, namely $\sat(\pset{T}) := \bases{\pset{T}}:I^{\infty}$, where $I = \prod_{T \in \pset{T}}\ini(T)$.

Consider the lexicographic term order $<_{\rm lex}$ induced by $x_1 < \cdots < x_n$. For an ideal $\ideal{I} \subseteq \fk[x_1, \ldots, x_n]$, let $\pset{G}$ be the reduced \grobner basis of $\ideal{I}$ w.r.t. $<_{\rm lex}$. Then the minimal triangular set contained in $\pset{G}$ can be extracted in the following way. For $i \in [n]$, choose from $\pset{G}_i := \{G\in \pset{G}: \lv(G) = x_i\}$ the minimal polynomial $C_i$ w.r.t. $<_{\rm lex}$ ($C_i$ is set to null if $\pset{G}_i = \emptyset$). Then the triangular set $\pset{C} = [C_1, \ldots, C_n]$ is called the \emph{W-characteristic set} of $\ideal{I}$. Note that $\pset{C}$ does not necessarily contain $n$ polynomials since $\pset{G}_i$ may be empty. From the uniqueness of the reduced \grobner basis, that of the W-characteristic set follows. W-characteristic sets serve as an underlying bridge connecting the theories of \grobner bases and triangular sets. 

For any polynomial ideal $\ideal{I} \subseteq \fk[x_1, \ldots, x_n]$, the pair $(\pset{G}, \pset{C})$ is called its \emph{characteristic pair}, where $\pset{G}$ and $\pset{C}$ are the reduced \grobner basis and W-characteristic set of $\ideal{I}$ respectively. When $\pset{C}$ is normal, we say that the characteristic pair is \emph{normal}; when $\bases{\pset{G}} = \sat(\pset{C})$, it is said to be \emph{strong}. Normality and being strong are two desired properties of characteristic pairs, for the former is easy to define yet still implies the fundamental regularity of the W-characteristic set and the latter provides two kinds of representations of the ideal in terms of the \grobner basis and triangular set simultaneously. Characteristic pairs are the underlying concept in the theory of characteristic decomposition developed in \cite{WDM20d}, which generalizes the traditional theory of triangular decomposition by combining the advantages of both \grobner bases and triangular sets. 

\subsection{Schubert determinantal ideal}
\label{sec:pre-schubert}

Let $k$ be a positive integer. For the permutations in the symmetric group $S_k$, we use the one-line notation to write them, namely $w=w_1 \cdots w_k$ denotes the permutation sending $i$ to $w_i$ for $i \in [k]$. For any $w = w_1\cdots w_k \in S_k$, its corresponding permutation matrix $w^T$ is the square matrix of size $k$ such that its $i$th row has 1 in the $w_i$-th column and 0 in all the other columns for $i \in [k]$.

Let $X$ be a generic matrix of size $m \times n$ with entries $x_{ij}$ for $i \in [m]$ and $j \in [n]$. Without loss of generality, we assume $m \leq n$. In this paper we work on determinantal ideals in the polynomial ring $\fk[\bm{x}]$ defined by minors of $X$ satisfying certain conditions, where $\bm{x}$ is the set of all the $mn$ variables $x_{11}, \ldots, x_{1n}, \ldots, x_{m1}, \ldots, x_{mn}$ appearing in $X$ as its entries. Let $M$ be an arbitrary matrix of size $m \times n$. Then for a pair $(p, q)$ of integers with $p \in [m]$ and $q \in [n]$, we use $M_{pq}$ to denote the submatrix of $M$ consisting its northwest $p\times q$ entries. 

For the Schubert determinantal ideals in this subsection, we work on permutations in $S_n$ for a fixed positive integer $n$ and the generic matrix $X$ of size $n \times n$.

\begin{definition}\rm \label{def:schubert}
  Let $w$ be a permutation in $S_n$. Then the \emph{Schubert determinantal ideal} of $w$ in the ring $\fk[\bm{x}]$, denoted by $I_w$, is the ideal generated by the minors of $X_{pq}$ of size $\rank(w^{T}_{pq})+1$ for all $p, q \in [n]$. 
\end{definition}

In \cite{FUL92F} it is shown that a smaller set of generators, called the \emph{Fulton generators} now in the literature, can be found as explained below. For an arbitrary permutation $w = w_1\cdots w_n \in S_n$, consider a square of $n \times n$ boxes and label them in the same way as entries of a square matrix. Then for each $i \in [n]$, remove the box at the position $(i, w_i)$ and all the boxes to the south and east of it. The remaining boxes form the \emph{Rothe diagram} $D(w)$ of $w$, and the \emph{essential set} of $w$ is defined as
$$\ess(w) := \{(p, q) \in D(w): \mbox{neither } (p, q+1) \mbox{ nor } (p+1, q) \mbox{ is in } D(w)\}.$$
Then instead of all $(p, q)$ in Definition~\ref{def:schubert}, it is proved that $I_w$ can be generated as an ideal by the minors of $X_{pq}$ of size $\rank(w^{T}_{pq})+1$ for all $(p, q) \in \ess(w)$.

A permutation $w \in S_n$ is said to be \emph{vexillary} (or 2143-avoiding) if there does not exist $1 \leq i < j < k < \ell \leq n$ such that $w_j < w_i < w_{\ell} < w_k$. It can be shown, see \cite[Remark~9.17]{FUL92F} for example, that a permutation $w$ is vexillary if and only if there are no $(p, q), (p', q') \in \ess(w)$ such that $p < p'$ and $q < q'$.

\begin{example}\rm \label{exp:ess}
  Consider the permutation $w=[10, 9, 2, 3, 8, 6, 5, 7, 4, 1] \in S_{10}$. Then the left subfigure in Figure~\ref{fig:EssExp} demonstrates which boxes are removed by lines to the south and east of $(i, w_i)$ (marked with $*$) and the essential set (marked with circles) in the Rothe diagram. In the right subfigure for each $(p,q) \in \ess(w)$, the rank $\rank(w^{T}_{pq})$, which is equal to the number of $*$ in the boxes to the north and west of $(p,q)$, is shown inside the circles. One can see that $w$ is vexillary either from $w$ itself or from the distribution of $\ess(w)$ . 
\end{example}

  \begin{figure}[h]
    \centering

    \begin{subfigure}{.48\textwidth}
      \centering
\begin{tikzpicture}
            \begin{scope}[scale =0.5]
                \draw[help lines,color=black] (0,0) grid (10,10);
                \node[black] at (9.5,9.5){$\ast$};
                \node[black] at (8.5,8.5){$\ast$};
                \node[black] at (1.5,7.5){$\ast$};
                \node[black] at (2.5,6.5){$\ast$};
                \node[black] at (7.5,5.5){$\ast$};
                \node[black] at (5.5,4.5){$\ast$};
                \node[black] at (4.5,3.5){$\ast$};
                \node[black] at (6.5,2.5){$\ast$};
                \node[black] at (3.5,1.5){$\ast$};
                \node[black] at (0.5,0.5){$\ast$};
                \draw[red](0.5,0)--(0.5,0.5)--(10,0.5);
                \draw[red](9.5,0)--(9.5,9.5)--(10,9.5);
                \draw[red](8.5,0)--(8.5,8.5)--(10,8.5);
                \draw[red](1.5,0)--(1.5,7.5)--(10,7.5);
                \draw[red](2.5,0)--(2.5,6.5)--(10,6.5);
                \draw[red](7.5,0)--(7.5,5.5)--(10,5.5);
                \draw[red](5.5,0)--(5.5,4.5)--(10,4.5);
                \draw[red](4.5,0)--(4.5,3.5)--(10,3.5);
                \draw[red](6.5,0)--(6.5,2.5)--(10,2.5);
                \draw[red](3.5,0)--(3.5,1.5)--(10,1.5);
                \draw[orange] (0.5,1.5) circle(0.45);
                \draw[orange] (3.5,2.5) circle(0.45);
                \draw[orange] (4.5,4.5) circle(0.45);
                \draw[orange] (6.5,5.5) circle(0.45);
                \draw[orange] (7.5,8.5) circle(0.45);
                \draw[orange] (8.5,9.5) circle(0.45);
            \end{scope}
        \end{tikzpicture}
      \end{subfigure}      
            \begin{subfigure}{.48\textwidth}
      \centering
        \begin{tikzpicture}
            \begin{scope}[scale =0.5]
                \draw[help lines,color=black] (0,0) grid (10,10);

                \foreach \x in {0,...,9}{
                    \filldraw[fill=gray!20] (\x,1) rectangle (\x +1,0);}
                \foreach \x in {3,...,9}{
                    \filldraw[fill=gray!20] (\x,2) rectangle (\x +1,1);}
                \foreach \x in {6,...,9}{
                    \filldraw[fill=gray!20] (\x,3) rectangle (\x +1,2);}
                \foreach \x in {4,...,9}{
                    \filldraw[fill=gray!20] (\x,4) rectangle (\x +1,3);}
                \foreach \x in {5,...,9}{
                    \filldraw[fill=gray!20] (\x,5) rectangle (\x +1,4);}
                \foreach \x in {7,...,9}{
                    \filldraw[fill=gray!20] (\x,6) rectangle (\x +1,5);}
                \foreach \x in {2,...,9}{
                    \filldraw[fill=gray!20] (\x,7) rectangle (\x +1,6);}
                \foreach \x in {1,...,9}{
                    \filldraw[fill=gray!20] (\x,8) rectangle (\x +1,7);}
                \foreach \x in {8,...,9}{
                    \filldraw[fill=gray!20] (\x,9) rectangle (\x +1,8);}
                \foreach \x in {1,...,7}{
                    \filldraw[fill=gray!20] (1,\x) rectangle (2,\x+1);}
                \foreach \x in {1,...,6}{
                    \filldraw[fill=gray!20] (2,\x) rectangle (3,\x+1);}
                \foreach \x in {1,...,3}{
                    \filldraw[fill=gray!20] (4,\x) rectangle (5,\x+1);}
                \foreach \x in {1,...,4}{
                    \filldraw[fill=gray!20] (5,\x) rectangle (6,\x+1);}
                \filldraw[fill=gray!20] (9,9) rectangle (10,10);
                
                \node[black] at (9.5,9.5){$\ast$};
                \node[black] at (8.5,8.5){$\ast$};
                \node[black] at (1.5,7.5){$\ast$};
                \node[black] at (2.5,6.5){$\ast$};
                \node[black] at (7.5,5.5){$\ast$};
                \node[black] at (5.5,4.5){$\ast$};
                \node[black] at (4.5,3.5){$\ast$};
                \node[black] at (6.5,2.5){$\ast$};
                \node[black] at (3.5,1.5){$\ast$};
                \node[black] at (0.5,0.5){$\ast$};

                \draw[orange] (0.5,1.5) circle(0.45) node[font=\fontsize{6}{6}\selectfont,black] {0};
                \draw[orange] (3.5,2.5) circle(0.45) node[font=\fontsize{6}{6}\selectfont,black] {2};
                \draw[orange] (4.5,4.5) circle(0.45) node[font=\fontsize{6}{6}\selectfont,black] {2};
                \draw[orange] (6.5,5.5) circle(0.45) node[font=\fontsize{6}{6}\selectfont,black] {2};
                \draw[orange] (7.5,8.5) circle(0.45) node[font=\fontsize{6}{6}\selectfont,black] {0};
                \draw[orange] (8.5,9.5) circle(0.45) node[font=\fontsize{6}{6}\selectfont,black] {0};
            \end{scope}
          \end{tikzpicture}
        \end{subfigure}
        \caption{An illustrative example of the Rothe diagram and essential set of the permutation $w=[10, 9, 2, 3, 8, 6, 5, 7, 4, 1]$}\label{fig:EssExp}
\end{figure}
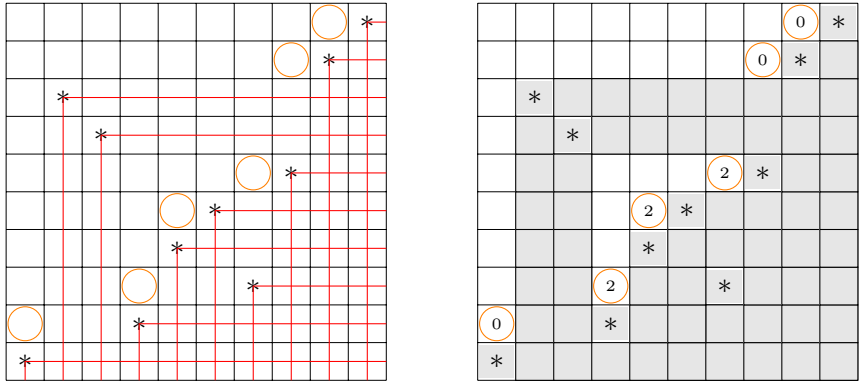

For an arbitrary Schubert determinantal ideal in $\kx$, all of its generators defined above are minors of $X$ and each generator is a homogeneous polynomial in $\kx$ of degree $k$ if it is (the expansion of) the determinant of a $k\times k$ submatrix. For determinantal ideals like the Schubert ones, one is particularly interested in two kinds of term orders related to determinants. A term order in $\kx$ is said to be \emph{(anti-)diagonal} if the leading term of any minor of $X$ is the product of all the (anti-)diagonal elements of the corresponding submatrix. The readers are referred to \cite[Chapter 16.4]{MS05C} for examples of (anti-)diagonal term orders in $\kx$. For the (anti-)diagonal term orders, the \grobner bases of Schubert determinantal ideals are known. 

\begin{theorem}[{\cite{KM05g,KMY09g}}]\label{thm:gb-schubert}
  Let $w$ be a permutation in $S_n$. Then the following statements hold.
  \begin{enumerate}
  \item[\rm (1)] Fulton generators form a \grobner basis of $I_w$ w.r.t. any anti-diagonal term order.
  \item[\rm (2)] Fulton generators form a \grobner basis of $I_w$ w.r.t. any diagonal term order if and only if $w$ is vexillary. 
  \end{enumerate}
\end{theorem}

Besides identification of Fulton generators as the \grobner basis of $I_w$ in Theorem~\ref{thm:gb-schubert}, for the anti-diagonal term order deep underlying relations between the initial ideal $\lt(I_w)$, the Stanley-Reisner simplicial complex, the double Schubert polynomial of $w$, and all the reduced pipe dreams of $w$ are also established in \cite{KM05g}.

\subsection{Elusive minors as minimal \grobner bases}
\label{sec:elusive}

Theorem~\ref{thm:gb-schubert} states that Fulton generators form \grobner bases of Schubert determinantal ideals, but in general they are not minimal. In \cite{GY22M}, the concept of elusive minors is introduced to study the minimality of Fulton generators, and next we summarize the underlying ideas of the method, which will also shed light on our study on reduced \grobner bases. 

Let $w \in S_n$ be a permutation. Then the Fulton generators $\pset{F}$ of $I_w$ consist of all the minors of $X_{pq}$ of size $\rank(w^{T}_{pq})+1$ for all $(p, q) \in \ess(w)$. To study the minimal \grobner basis of $I_w$, by definition we need to investigate whether there exist two distinct minors $\p{m}_1, \p{m}_2 \in \pset{F}$ such that $\lt(\p{m}_1) | \lt(\p{m}_2)$, where the term order is either an anti-diagonal one or a diagonal one. Clearly if $\p{m}_1$ and $\p{m}_2$ belong to the same $(p, q) \in \ess(w)$, then they are the determinants of different submatrices of $X$ of the same size, and thus in this case the divisibility $\lt(\p{m}_1) | \lt(\p{m}_2)$ never occurs. Therefore it suffices to study the divisibility between minors from distinct $(p,q), (\tilde{p}, \tilde{q}) \in \ess(w)$.

Fix a minor $\p{m}$ from $(p, q) \in \ess(w)$ and another $(\tilde{p}, \tilde{q}) \in \ess(w)$, we want to test whether there exists a minor $\tilde{\p{m}}$ from  $(\tilde{p}, \tilde{q}) \in \ess(w)$  such that $\lt(\tilde{\p{m}}) | \lt(\p{m})$. It turns out that this can be verified by counting the number of full rows or columns of $\p{m}$ that lie in the submatrix $X_{\tilde{p}\tilde{q}}$. Denote by $R(\p{m})$ and $C(\p{m})$ respectively the sets of row and column indices of $\p{m}$. They are subsets of distinct integers in $[n]$ with the same cardinality. Then any minor $\p{m}$ of $X$ can be uniquely represented by $(R(\p{m}), C(\p{m}))$.

\begin{definition}\rm \label{def:attend}
Let $w \in S_n$ be a permutation, $\p{m}$ be a minor of size $\rank(w^{T}_{pq})$ $+1$ for some $(p, q) \in \ess(w)$, and $(\tilde{p}, \tilde{q}) \in \ess(w)$ be distinct from $(p, q)$. Then $\p{m}$ is said to \emph{attend} the submatrix $X_{\tilde{p}\tilde{q}}$ if (1) $\#(R(\p{m}) \cap [\tilde{p}]) \geq \rank(w^T_{\tilde{p}\tilde{q}})+1$ and $\#(C(\p{m})\cap [\tilde{q}]) = \rank(w^{T}_{pq})+1$, or (2) $\#(R(\p{m})\cap [\tilde{p}]) = \rank(w^{T}_{pq})+1$ and $\#(C(\p{m}) \cap [\tilde{q}]) \geq \rank(w^T_{\tilde{p}\tilde{q}})+1$. 
\end{definition}

Note that condition~(1) in Definition~\ref{def:attend} above means that the intersection of $\p{m}$ and the submatrix $X_{\tilde{p}\tilde{q}}$ contains at least $\rank(w^T_{\tilde{p}\tilde{q}})+1$ full rows of $\p{m}$. Since minors chosen from $X_{\tilde{p}\tilde{q}}$ for $(\tilde{p}, \tilde{q}) \in \ess(w)$ are of size $\rank(w^T_{\tilde{p}\tilde{q}})+1$, it is obvious that in this case there exists a minor $\tilde{\p{m}}$ from $X_{\tilde{p}\tilde{q}}$ such that $\lt(\tilde{\p{m}}) | \lt(\p{m})$. Similarly, condition~(2) is for the case of columns.

\begin{example}\rm \label{exp:attend}
Let us continue with the permutation $w$ in Example~\ref{exp:ess}. Choose $(5, 7), (2, 8) \in \ess(w)$, and the outlines of $X_{57}$ and $X_{28}$ are drawn with red and blue lines respectively in Figure~\ref{fig:attend}. In the left subfigure, the shaded minor $\p{m} = (\{2,3,5\}, \{2,4,6\})$ from $X_{57}$ is of size $3 = \rank(w_{57}^T)+1$ and thus is a Fulton generator of $I_w$. It is clear that $\p{m}$ intersects $X_{28}$ with $1$ full row and thus $\p{m}$ attends $X_{28}$. Then $\tilde{\p{m}}_a = x_{26}$ divides $\lt(\p{m}) = x_{26}x_{34}x_{52}$ for the anti-diagonal term order and $\tilde{\p{m}}_d = x_{22}$ divides $\lt(\p{m}) = x_{22}x_{34}x_{56}$ for the diagonal term order. In the right subfigure, the shaded minor $(\{3,4,5\}, \{3,5,6\})$ does not attend $X_{28}$, and its leading term w.r.t. either an anti-diagonal or a diagonal term order is not divisible by any minor of size $1$ in $X_{28}$. 
\end{example}

  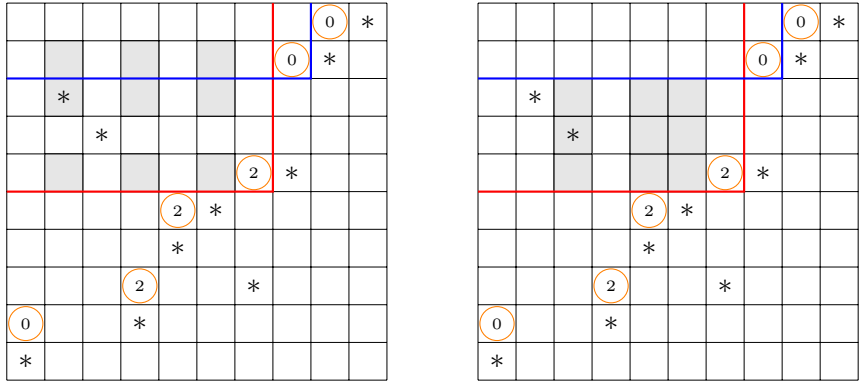
\begin{figure}[h]
    \centering

    \begin{subfigure}[b]{.48\textwidth}
      \centering
\begin{tikzpicture}
       \begin{scope}[scale =0.5]
            \draw[help lines,color=black] (0,0) grid (10,10);
            \filldraw[fill=gray!20] (1,6) rectangle (2,5);
            \filldraw[fill=gray!20] (1,8) rectangle (2,7);
            \filldraw[fill=gray!20] (1,9) rectangle (2,8);
            \filldraw[fill=gray!20] (3,6) rectangle (4,5);
            \filldraw[fill=gray!20] (3,8) rectangle (4,7);
            \filldraw[fill=gray!20] (3,9) rectangle (4,8);
            \filldraw[fill=gray!20] (5,6) rectangle (6,5);
            \filldraw[fill=gray!20] (5,8) rectangle (6,7);
            \filldraw[fill=gray!20] (5,9) rectangle (6,8);

            \node[black] at (9.5,9.5){$\ast$};
            \node[black] at (8.5,8.5){$\ast$};
            \node[black] at (1.5,7.5){$\ast$};
            \node[black] at (2.5,6.5){$\ast$};
            \node[black] at (7.5,5.5){$\ast$};
            \node[black] at (5.5,4.5){$\ast$};
            \node[black] at (4.5,3.5){$\ast$};
            \node[black] at (6.5,2.5){$\ast$};
            \node[black] at (3.5,1.5){$\ast$};
            \node[black] at (0.5,0.5){$\ast$};
            \draw[orange] (0.5,1.5) circle(0.45) node[font=\fontsize{6}{6}\selectfont,black] {0};
            \draw[orange] (3.5,2.5) circle(0.45) node[font=\fontsize{6}{6}\selectfont,black] {2};
            \draw[orange] (4.5,4.5) circle(0.45) node[font=\fontsize{6}{6}\selectfont,black] {2};
            \draw[orange] (6.5,5.5) circle(0.45) node[font=\fontsize{6}{6}\selectfont,black] {2};
            \draw[orange] (7.5,8.5) circle(0.45) node[font=\fontsize{6}{6}\selectfont,black] {0};
            \draw[orange] (8.5,9.5) circle(0.45) node[font=\fontsize{6}{6}\selectfont,black] {0};
            \draw[thick, red](0,5)--(7,5)--(7,10);
            \draw[thick, blue](0,8)--(8,8)--(8,10);

        \end{scope}
      \end{tikzpicture}
    \end{subfigure}
            \begin{subfigure}{.48\textwidth}
              \centering
              \begin{tikzpicture}
        \begin{scope}[scale =0.5]
            \draw[help lines,color=black] (0,0) grid (10,10);

            \filldraw[fill=gray!20] (2,6) rectangle (3,5);
            \filldraw[fill=gray!20] (2,7) rectangle (3,6);
            \filldraw[fill=gray!20] (2,8) rectangle (3,7);
            \filldraw[fill=gray!20] (4,6) rectangle (5,5);
            \filldraw[fill=gray!20] (4,7) rectangle (5,6);
            \filldraw[fill=gray!20] (4,8) rectangle (5,7);
            \filldraw[fill=gray!20] (5,6) rectangle (6,5);
            \filldraw[fill=gray!20] (5,7) rectangle (6,6);
            \filldraw[fill=gray!20] (5,8) rectangle (6,7);

            \node[black] at (9.5,9.5){$\ast$};
            \node[black] at (8.5,8.5){$\ast$};
            \node[black] at (1.5,7.5){$\ast$};
            \node[black] at (2.5,6.5){$\ast$};
            \node[black] at (7.5,5.5){$\ast$};
            \node[black] at (5.5,4.5){$\ast$};
            \node[black] at (4.5,3.5){$\ast$};
            \node[black] at (6.5,2.5){$\ast$};
            \node[black] at (3.5,1.5){$\ast$};
            \node[black] at (0.5,0.5){$\ast$};
            \draw[orange] (0.5,1.5) circle(0.45) node[font=\fontsize{6}{6}\selectfont,black] {0};
            \draw[orange] (3.5,2.5) circle(0.45) node[font=\fontsize{6}{6}\selectfont,black] {2};
            \draw[orange] (4.5,4.5) circle(0.45) node[font=\fontsize{6}{6}\selectfont,black] {2};
            \draw[orange] (6.5,5.5) circle(0.45) node[font=\fontsize{6}{6}\selectfont,black] {2};
            \draw[orange] (7.5,8.5) circle(0.45) node[font=\fontsize{6}{6}\selectfont,black] {0};
            \draw[orange] (8.5,9.5) circle(0.45) node[font=\fontsize{6}{6}\selectfont,black] {0};
            \draw[thick, red](0,5)--(7,5)--(7,10);
            \draw[thick, blue](0,8)--(8,8)--(8,10);
        \end{scope}
      \end{tikzpicture}
    \end{subfigure}
    
        \caption{An illustrative example of a Fulton generator attending a submatrix} \label{fig:attend}
\end{figure}
      
For the easier case as in Example~\ref{exp:attend} when $p > \tilde{p}$ and $q > \tilde{q}$ do not hold simultaneously for $(p, q), (\tilde{p}, \tilde{q}) \in \ess(w)$, divisibility of $\lt(\p{m})$ for a Fulton generator $\p{m}$ from $X_{pq}$ by the leading term of some minor from $X_{\tilde{p}\tilde{q}}$ is fully characterized by whether $\p{m}$ attends $X_{\tilde{p}\tilde{q}}$ or not. In \cite{GY22M} it is proved that for the anti-diagonal term order, the above characterization also holds for the trickier case when $p > \tilde{p}$ and $q > \tilde{q}$ do hold (for which the intersection does not necessarily consists of \emph{full} rows or columns).

\begin{proposition}[{\cite{GY22M}}]\label{prop:attend}
  Let $w \in S_n$ be a permutation, $\p{m}$ be a Fulton generator for $(p, q) \in \ess(w)$, and $(\tilde{p}, \tilde{q})$ be another pair in $\ess(w)$.
  \begin{enumerate}
  \item[\rm (1)] For any anti-diagonal term order, there exists a Fulton generator $\tilde{\p{m}}$ for $(\tilde{p}, \tilde{q})$ such that $\lt(\tilde{\p{m}}) | \lt(\p{m})$ if and only if $\p{m}$ attends $X_{\tilde{p}\tilde{q}}$.
  \item[\rm (2)] If $w$ is vexillary, then for any diagonal term order, the same as in {\rm (1)} also holds. 
  \end{enumerate}
\end{proposition}

A Fulton generator of $I_w$ for some $(p, q) \in \ess(w)$ is said to be \emph{elusive} if it does not attend any $X_{\tilde{p}\tilde{q}}$ for each $(\tilde{p}, \tilde{q}) \in \ess(w)\setminus \{(p, q)\}$ such that $\rank(w_{\tilde{p}\tilde{q}}^T) < \rank(w_{pq}^T)$. Combining Proposition~\ref{prop:attend} with Theorem~\ref{thm:gb-schubert}, one immediately gets the following descriptions about the minimal \grobner bases of Schubert determinantal ideals.

\begin{theorem}[{\cite[Corollary~1.8]{GY22M}}]\label{thm:minimal}
  Let $w$ be a permutation in $S_n$. Then the following statements hold.
  \begin{enumerate}
  \item[\rm (1)] For any anti-diagonal term order, all the elusive minors form a minimal \grobner basis of $I_w$.
  \item[\rm (2)] If $w$ is vexillary, then for any diagonal term order, the same as in {\rm (1)} also holds. 
  \end{enumerate}
\end{theorem}

The following proposition about elusive minors is useful in our study on the reduced \grobner basis of $I_w$ in Section~\ref{sec:reduced}. Clearly the minor specified in this proposition is of size $\rank(w^T_{pq})+1$ and $(p, q)$ is its southeast corner.

\begin{proposition}[{\cite[Proposition~2.2]{GY22M}}]\label{prop:corner}
  Let $(p, q)$ be a pair in $D(w)$ for a permutation $w \in S_n$ and $r = \rank(w^T_{pq})$. Then the minor $([p-r, p], [q-r, q])$ is an elusive one of $I_w$. 
\end{proposition}

\subsection{Ladder determinantal ideal}
\label{sec:pre-ladder}

In this subsection $X$ is a generic matrix of size $m\times n$ with entries $x_{11}, \ldots, x_{mn}$. Ladder determinantal ideals are those generated by minors inside a special subset, called a ladder due to its shape, of $X$ in $\kx$, where $\p{x} = \{x_{11}, \ldots, x_{mn}\}$.

A (two-sided) \emph{ladder} in $X$ is a subset $L$ of $X$ such that whenever $x_{ij}, x_{k\ell} \in L$ with $i\leq k$ and $j \leq \ell$, we have $x_{pq} \in L$ for all $i \leq p \leq k$ and $j \leq q \leq \ell$. In other words, if the variables at the northwest and southeast corners of a submatrix of $X$ belong to a ladder $L$, then all the variables in this submatrix are in $L$ too. Note that the definition of ladders here is different from that in \cite{GOR07M} in the sense that our ladder is obtained by rotating the one there clockwisely by 90 degrees. This adjustment we make is to accord the directions of (anti-)diagonal blocks in the blockwise determinantal ideals with the (anti-)diagonal term orders, as the reader will soon see in Section~\ref{sec:block}. With this adjustment the one-one correspondence between one-sided ladder determinantal ideals and Schubert determinantal ideals for vexillary permutations, which is recalled later as Theorem~\ref{thm:one-sided-schubert}, also becomes more natural. 

Any ladder in $X$ can be specified by four sequences of integer indices. Let
\begin{equation}
  \label{eq:ladder-seq}
  \begin{split}
    &1 \leq a_1 < \cdots < a_{\ell} \leq m, \quad n \geq b_1 \geq \cdots \geq b_{\ell} \geq 1\\
    &1 \leq c_1 < \cdots < c_u \leq m, \quad n \geq d_1 \geq \cdots \geq d_u \geq 1.
  \end{split}
\end{equation}
Then a ladder $L$ in $X$ can be defined as
$$ L = \{x_{pq} |~ \exists i, j \mbox{ such that } 1 \leq p \leq a_i, 1 \leq q \leq b_i, c_j\leq p \leq m, d_j \leq q \leq n\}.$$
And for this ladder $L$, the variables $x_{a_1b_1}, \ldots, x_{a_{\ell}b_{\ell}}$ are called its \emph{lower corners} and $x_{c_1d_1}, \ldots, x_{c_ud_u}$ called its \emph{upper corners}.

\begin{example}\rm
  Consider a generic square matrix $X$ of size $9$. The left subfigure in Figure~\ref{fig:ladder} demonstrates a two-sided ladder in $X$ with the lower corners $x_{19}, x_{28}, x_{57}, x_{65}, x_{84}, x_{91}$ and upper corners $x_{16}, x_{34}, x_{42}, x_{61}$ (all marked as shaded boxes). 
\end{example}

  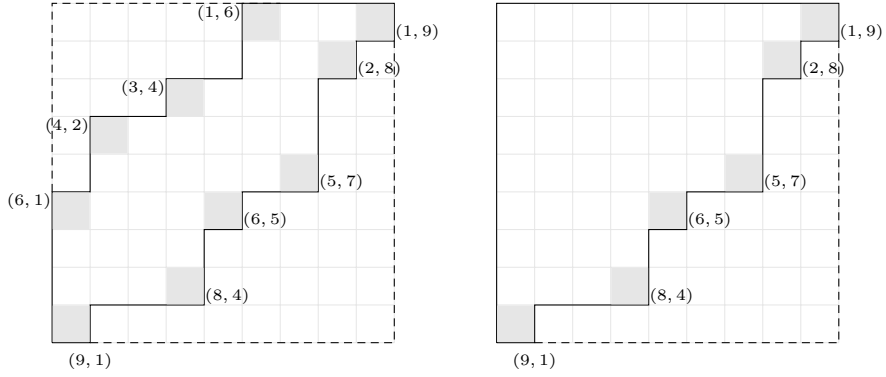
\begin{figure}[h]
    \centering
            \begin{subfigure}{.48\textwidth}
      \centering

      \begin{tikzpicture}
         \begin{scope}[scale =0.5]
                \draw[help lines,color=gray!20] (0,0) grid (9,9);
                \draw[black] (0,0)--(1,0)--(1,1)--(4,1)--(4,3)--(5,3)--(5,4)--(7,4)--(7,7)--(8,7)--(8,8)--(9,8)--(9,9)--(5,9)--(5,7)--(3,7)--(3,6)--(1,6)--(1,4)--(0,4)--(0,0);
   
                \node[font=\fontsize{6}{6}\selectfont,below,black] at (1,0){$(9,1)$};
                \node[font=\fontsize{6}{6}\selectfont,right,black] at (3.8,1.25){$(8,4)$};
                \node[font=\fontsize{6}{6}\selectfont,right,black] at (4.8,3.25){$(6,5)$};
                \node[font=\fontsize{6}{6}\selectfont,right,black] at (6.8,4.25){$(5,7)$};
                \node[font=\fontsize{6}{6}\selectfont,right,black] at (7.8,7.25){$(2,8)$};
                \node[font=\fontsize{6}{6}\selectfont,right,black] at (8.8,8.25){$(1,9)$};
                \node[font=\fontsize{6}{6}\selectfont,left,black] at (5.2,8.75){$(1,6)$};
                \node[font=\fontsize{6}{6}\selectfont,left,black] at (3.2,6.75){$(3,4)$};
                \node[font=\fontsize{6}{6}\selectfont,left,black] at (1.2,5.75){$(4,2)$};
                \node[font=\fontsize{6}{6}\selectfont,left,black] at (0.2,3.75){$(6,1)$};

                \filldraw[gray!20] (0.03,0.03) rectangle (0.97,0.97);
                \filldraw[gray!20] (3.03,1.03) rectangle (3.97,1.97);
                \filldraw[gray!20] (4.03,3.03) rectangle (4.97,3.97);
                \filldraw[gray!20] (6.03,4.03) rectangle (6.97,4.97);
                \filldraw[gray!20] (7.03,7.03) rectangle (7.98,7.97);
                \filldraw[gray!20] (8.03,8.03) rectangle (8.97,8.97);
                \filldraw[gray!20] (0.03,3.03) rectangle (0.97,3.97);
                \filldraw[gray!20] (1.03,5.03) rectangle (1.97,5.97);
                \filldraw[gray!20] (3.03,6.03) rectangle (3.97,6.97);
                \filldraw[gray!20] (5.03,8.03) rectangle (5.97,8.97);

                \draw[densely dashed,black](0,4)--(0,9)--(6,9);
                \draw[densely dashed,black](1,0)--(9,0)--(9,8);

            \end{scope}
          \end{tikzpicture}
        \end{subfigure}
            \begin{subfigure}{.48\textwidth}
      \centering

           \begin{tikzpicture} \begin{scope}[scale =0.5]
                \draw[help lines,color=gray!20] (0,0) grid (9,9);
                \draw[black] (0,0)--(1,0)--(1,1)--(4,1)--(4,3)--(5,3)--(5,4)--(7,4)--
                (7,7)--(8,7)--(8,8)--(9,8)--(9,9)--(0,9)--(0,0);

                \node[font=\fontsize{6}{6}\selectfont,below,black] at (1,0){$(9,1)$};
                \node[font=\fontsize{6}{6}\selectfont,right,black] at (3.8,1.25){$(8,4)$};
                \node[font=\fontsize{6}{6}\selectfont,right,black] at (4.8,3.25){$(6,5)$};
                \node[font=\fontsize{6}{6}\selectfont,right,black] at (6.8,4.25){$(5,7)$};
                \node[font=\fontsize{6}{6}\selectfont,right,black] at (7.8,7.25){$(2,8)$};
                \node[font=\fontsize{6}{6}\selectfont,right,black] at (8.8,8.25){$(1,9)$};

                \filldraw[gray!20] (0.03,0.03) rectangle (0.97,0.97);
                \filldraw[gray!20] (3.03,1.03) rectangle (3.97,1.97);
                \filldraw[gray!20] (4.03,3.03) rectangle (4.97,3.97);
                \filldraw[gray!20] (6.03,4.03) rectangle (6.97,4.97);
                \filldraw[gray!20] (7.03,7.03) rectangle (7.97,7.97);
                \filldraw[gray!20] (8.03,8.03) rectangle (8.97,8.97);
                \draw[densely dashed,black](1,0)--(9,0)--(9,8);
            \end{scope}
          \end{tikzpicture}
                  \end{subfigure}

          \caption{Two-sided and one-sided ladders in a generic square matrix of size 9}\label{fig:ladder}
\end{figure}

A ladder $L$ in $X$ is said to be \emph{one-sided} if all its upper corners lie in the first row (or column) of $L$. Take the northwest-most upper corner $c$ of $L$, then $L$, as a special two-sided ladder, can be represented by $c$ and all its lower corners. In particular, when $c = x_{11}$, the one-sided ladder $L$ is said to be \emph{full} and can be specified by two sequences of integers to represent all its lower corners (see the right subfigure in Figure~\ref{fig:ladder} for an example). One-sided ladder determinantal ideals defined below are generated by minors of specified sizes inside full one-sided ladders. 

\begin{definition}\rm\label{def:one-sided-ideal}
Let $X$ be a generic matrix of size $m \times n$. Let $\p{a} = (a_1, \ldots, a_k)$ and $\p{b} = (b_1, \ldots, b_k)$ be two integer sequences of length $k$ such that
  \begin{equation}
    \label{eq:one-sided-ab}
    1 \leq a_1 \leq a_2 \leq \cdots \leq a_k \leq m \mbox{ and } n \geq b_1 \geq b_2 \geq \cdots \geq b_k\geq 1
  \end{equation}
  and $\p{r} = (r_1, \ldots, r_k)$ be a sequence of positive integers such that
  \begin{equation}
    \label{eq:one-sided-r}
  0<a_1-r_1<a_2-r_2<\cdots <a_k-r_k \mbox{ and } b_1-r_1>b_2-r_2>\cdots>b_k-r_k>0.  
  \end{equation}
Then the \emph{one-sided ladder determinantal ideal} $I(\p{a},\p{b},\p{r})$ is the ideal in $\kx$ generated by all the minors of size $r_i$ in $X_{a_ib_i}$ for all $i = 1,\ldots, k$.
\end{definition}

One can check that conditions~\eqref{eq:one-sided-r} are to remove straightforward redundant generators: if $a_i -r_i \geq a_{i+1}-r_{i+1}$ for some $i$, then each $r_{i+1}$-minor in $X_{a_{i+1} b_{i+1}}$ can be represented by the $r_i$-minors in $X_{a_ib_i}$ with the generalized Laplace expansion and thus is redundant as a generator of the ideal.

The following fundamental correspondence between one-sided ladder determinantal ideals and Schubert determinantal ideals for vexillary permutations is established in \cite{FUL92F}.

\begin{theorem}[{\cite[Proposition~9.6]{FUL92F}}]\label{thm:one-sided-schubert}
  Let $I(\p{a},\p{b},\p{r})$ be a one-sided ladder determinantal ideal with $\p{a}$, $\p{b}$, and $\p{r}$ satisfying conditions \eqref{eq:one-sided-ab} and \eqref{eq:one-sided-r}. Then for each $n \geq a_k+b_1$, there exists a unique vexillary permutation $w\in S_n$ such that
  \begin{equation}
    \label{eq:one-sided-ess}
    \ess(w) = \{(a_1, b_1), \ldots, (a_k, b_k)\} \mbox{ and } \rank(w^T_{a_ib_i}) = r_i-1 \mbox{ for } i=1, \ldots, k.
  \end{equation}
   Furthermore, for each vexillary permutation there exist unique $\p{a}$, $\p{b}$, and $\p{r}$ satisfying
  \eqref{eq:one-sided-ab}, \eqref{eq:one-sided-r}, and \eqref{eq:one-sided-ess}.
\end{theorem}

As one can find, condition~\eqref{eq:one-sided-ess} in Theorem~\ref{thm:one-sided-schubert} above implies that the Fulton generators of the Schubert determinantal ideal $I_w$ for the unique vexillary permutation $w \in S_n$ in the theorem, once $n$ is fixed, coincide with the generators of $I(\p{a},\p{b},\p{r})$ as defined in Definition~\ref{def:one-sided-ideal}. In particular, for any one-sided ladder determinantal ideal, an effective method to find the corresponding vexillary permutation in Theorem~\ref{thm:one-sided-schubert} is also presented in the proof of \cite[Proposition~9.6]{FUL92F}. 

Combining Theorem~\ref{thm:gb-schubert} with the correspondence in Theorem~\ref{thm:one-sided-schubert}, the \grobner bases of one-sided ladder determinantal ideals can be obtained immediately. 

\begin{theorem}\label{thm:gb-one-sided}
Let $I(\p{a},\p{b},\p{r})$ be a one-sided ladder determinantal ideals with $\p{a}$, $\p{b}$, and $\p{r}$ satisfying the conditions in \eqref{eq:one-sided-ab} and \eqref{eq:one-sided-r} and $w \in S_n$ for $n=a_k+b_1$ be the vexillary permutation in Theorem~\ref{thm:one-sided-schubert} corresponding to $I(\p{a},\p{b},\p{r})$. Then the Fulton generators of $I_w$ form a \grobner basis of $I(\p{a},\p{b},\p{r})$ w.r.t. any (anti-)diagonal term order. 
\end{theorem}

Let $L$ be a ladder which is specified by four sequences as in \eqref{eq:ladder-seq}. For $i=1, \ldots, u$, denote $L_i = \{x_{pq}\in L |~ c_i \leq p \leq m, d_i \leq q \leq n\}$. Then each $L_i$ is a one-sided ladder with the unique upper corner $x_{c_id_i}$ and $L = \bigcup_{i=1}^u L_i$. Consider the polynomial ring $\fk[L]$ in the variables from $L$ over $\fk$. For a positive integer $r_i$, denote by $I(L_i,r_i)$ the ideal in $\fk[L]$ generated by all the $r_i$-minors of $X$ whose all entries belong to $L_i$. Then the two-sided ladder determinantal ideal for $L$ is defined by taking the sum of all these ideals $I(L_i,r_i)$. 

\begin{definition}\rm \label{def:mixed}
  Let $X$ be a generic matrix of size $m \times n$, $L$ be a ladder in $X$ with the upper corners $c_1, \ldots, c_u \in \p{x}$, and $L = \bigcup_{i=1}^u L_i$ with each $L_i$ a one-sided ladder with the unique upper corner $c_i$ for $i=1, \ldots, u$. Then for a sequence $\p{r} = (r_1, \ldots, r_u)$ of positive integers, the \emph{two-sided ladder determinantal ideal} $I(L, \p{r})$ of $L$ w.r.t. $\p{r}$ is the ideal $I(L,\p{r}) = \sum_{i=1}^u I(L_i,r_i)$ in $\fk[L]$. 
\end{definition}

Instead of the term \emph{mixed} ladder determinantal ideal in \cite{GOR07M}, we choose to call it a \emph{two-sided} one in Definition~\ref{def:mixed}. This choice bears the potential possibility of generalization to $n$-sided ladder determinantal ideals by using $(n\!-\!1)$-sided ladders, as done with 1- and 2-sided ladder determinantal ideals above (viewing squares in Definition~\ref{def:one-sided-ideal} as 0-sided ladders).

For a lexicographic diagonal term order induced by a scanning variable order (see Section~\ref{sec:w-char} for more details), the \grobner basis of a two-sided ladder determinantal ideal can be determined as shown in the theorem below.

\begin{theorem}[{\cite[Theorem~1.10]{GOR07M}}]\label{thm:mixed}
  Let $I(L,\p{r})$ be a two-sided ladder determinantal ideal of $L = \bigcup_{i=1}^u L_i$ w.r.t. $\p{r}$. For each $i=1, \ldots, u$, let $\pset{G}_i$ be the set of all the $r_i$-minors with all entries in $L_i$ such that at most their $r_j-1$ columns belong to $L_j$ for $1\leq j <i$ and at most $r_j-1$ rows belong to $L_j$ for $i < j \leq u$. Then $\pset{G} = \bigcup _{i=1}^u \pset{G}_i$ is a \grobner basis of $I(L,\p{r})$ w.r.t. any lexicographic diagonal term order induced by a scanning variable order.
\end{theorem}

\section{Blockwise determinantal ideals and their \grobner bases}
\label{sec:block}

Let $X$ be a generic matrix of size $m \times n$ throughout this section. As one can see, the above-mentioned determinantal ideals are defined by generators of different sets of minors of $X$ such that for each set, the minors are of the same size and have entries from the same subset of $X$.

Let $B$ be a non-empty subset of $X$. Then $B$ is said to be \emph{anti-diagonal} in $X$ if whenever two variables $x_{pq}, x_{\tilde{p}\tilde{q}} \in B$ for $p \geq \tilde{p}$ and $q \leq \tilde{q}$, all the variables in $\{x_{ij}: \tilde{p}\leq i \leq p, q \leq j \leq \tilde{q}\}$ also belong to $B$. That is to say, if the northeast and southwest corners of a submatrix belong to $B$, all the entries in the submatrix do too. \emph{Diagonal} subsets can be similarly defined. Clearly, $X_{pq}$ for any $p \in [m]$ and $q \in [n]$ is both anti-diagonal and diagonal, while the ladders defined in Section~\ref{sec:pre-ladder} are only diagonal.

\begin{definition}\rm\label{def:block}
  Let $\p{B} = (B_1, \ldots, B_k)$ be a sequence of subsets of $X$ and $\p{r} = (r_1, \ldots, r_k)$ be a sequence of positive integers. Then the \emph{blockwise determinantal ideal} w.r.t. $\p{B}$ and $\p{r}$, denoted by $I(\p{B}, \p{r})$, is the ideal in $\kx$ generated by all the minors of size $r_i$ in $B_i$ of $X$ for all $i \in [k]$. The subsets $B_1, \ldots, B_k$ are called the \emph{blocks} of $I(\p{B}, \p{r})$. When all its blocks are anti-diagonal (resp. diagonal) in $X$, $I(\p{B}, \p{r})$ is said to be \emph{anti-diagonal} (resp. \emph{diagonal}). 
\end{definition}

Blockwise determinantal ideals are generalization of many determinantal ideals of interest as follows. 

\begin{enumerate}
\item For the determinantal ideal $I_r$ generated by all the $r$-minors in $X$, we have $I_r = I(X, r)$ with a single block. 
\item For the general determinantal ideal $I$ considered in \cite[Page~5]{BV06D} defined by four sequences
\begin{equation*}
  \begin{split}
    &1 \leq u_1 \leq \cdots \leq u_p \leq m, \quad 0 \leq r_1 < \cdots < r_p < m,\\
    &1 \leq v_1 \leq \cdots \leq v_q \leq n, \quad 0 \leq s_1 < \cdots < s_q < n.
  \end{split}
\end{equation*}
Let $\p{B} = (\p{R}_1, \ldots, \p{R}_p, \p{C}_1, \ldots, \p{C}_q)$ with $\p{R}_i$ consisting of the first $u_i$ rows of $X$ for $i\in [p]$ and $\p{C}_j$ of the first $v_j$ columns of $X$ for $j \in [q]$ and $\p{r} = (r_1+1, \ldots, r_p+1, s_1+1, \ldots, s_q+1)$. Then $I$ equals $I(\p{B}, \p{r})$ and it is both anti-diagonal and diagonal. 
\item For the Schubert determinantal ideal $I_w$, let $\p{B} = (X_{pq}: (p,q) \in \ess(w))$ and $\p{r} = (\rank(w^{T}_{pq})+1: (p,q)\in \ess(w))$. Then $I_w$ can be written as the blockwise determinantal ideal $I(\p{B}, \p{r})$, and it is both anti-diagonal and diagonal. 
\item For the one-sided determinantal ideal $I(\p{a}, \p{b}, \p{r})$ with $\p{a}$ and $\p{b}$ satisfying condition~\eqref{eq:one-sided-ab} and $\p{r}$ satisfying condition~\eqref{eq:one-sided-r}, it is straightforward to see that $I(\p{a}, \p{b}, \p{r}) = I(\p{B}, \p{r})$ with $B_i = X_{a_ib_i}$ for $i=1, \ldots, k$ and it is also both anti-diagonal and diagonal. 
\item For the two-sided determinantal ideal $I(L,\p{r})$ in Definition~\ref{def:mixed}, clearly one has $I(L,\p{r}) = I(\p{B}, \p{r})$ with $\p{B} = (L_1, \ldots, L_u)$ and it is only diagonal. To be precise, in fact we extend the ideal $I(L,\p{r})$ from $\fk[L]$ in Definition~\ref{def:mixed} to $\kx$. 
\end{enumerate}

In this paper we are only interested in anti-diagonal and diagonal blockwise determinantal ideals, and one will see in later sections that \grobner bases of blockwise determinantal ideals only make sense w.r.t. anti-diagonal (resp. diagonal) term orders for anti-diagonal (resp. diagonal) blockwise determinantal ideals: by simply checking the underlying term orders used for the \grobner bases for Schubert determinantal ideals in Section~\ref{sec:pre-schubert} and ladder determinantal ideals in Section~\ref{sec:pre-ladder}, one can already see this. Since diagonal subsets of $X$ are precisely ladders in $X$ defined in Section~\ref{sec:pre-ladder} and anti-diagonal subsets are precisely rotated ladders by 90 degrees, the blockwise determinantal ideals of our interest are actually constructed from ladders (or rotated ladders) in general, and thus they strictly contain all the above examples. 

The defining blocks of a general blockwise determinantal ideal can have complicated relative positions, and thus in general it is hard to conclude whether the generating sets of blockwise determinantal ideals form their \grobner bases. Next we present several sufficient conditions for this. The discussions are for \grobner bases of diagonal blockwise determinantal ideals and the underlying term orders are diagonal ones accordingly. The same results hold for \grobner bases of anti-diagonal blockwise determinantal ideals w.r.t. anti-diagonal term orders. The following theorem is the main tool in our discussions. 

\begin{theorem}[{\cite[Lemma 1.3.14]{C1993G}}]\label{thm:grobnerblcok}

    Let $I_1$ and $I_2$ be two homogeneous ideals in $\fk[x_1, \ldots, x_n]$ and $\pset{G}_1$ and $\pset{G}_2$ be their \grobner bases w.r.t. a term order $<$ respectively. Then $\pset{G}_1 \cup \pset{G}_2 $ form a \grobner basis of $I_1 + I_2$ w.r.t. $<$ if and only if for all $G_1 \in \pset{G}_1$ and $G_2 \in \pset{G}_2$, there exists a polynomial $F \in I_1\cap I_2$ such that $\lt(F)=\lcm(\lt(G_1),\lt(G_2))$.
\end{theorem}

In the sequel of this section, $I(\p{B}, \p{r})$ is a diagonal blockwise determinantal ideal in $\fk[\p{x}]$ with $\p{B}=(B_1,\ldots,B_k)$ and $\p{r}=(r_1,\ldots,r_k)$, $\pset{G}_i$ is the set of all the $r_i$-minors in $B_i$ and $I_i = \bases{\pset{G}_i} \subseteq \fk[\p{x}]$ for each $i\in [k]$, and $<$ is an arbitrary lexicographic diagonal term order induced by a scanning variable order. Then we know that $I(\p{B}, \p{r})= \sum_{i=1}^k I_i$ with each $I_i$ a homogeneous ideal in $\fk[\p{x}]$ and $\pset{G}_i$ a \grobner basis of $I_i$ w.r.t. $<$ for $i \in [k]$. 

\begin{proposition}\label{prop:disjointleadingterm}
If for any $G_i \in \pset{G}_i$ and $G_j \in \pset{G}_j$ with $i\neq j$, the variables in $\lt(G_i)$ are completely different from those in $\lt(G_j)$, then $\bigcup_{i=1}^k \pset{G}_i$ is a \grobner basis of $I(\p{B}, \p{r})$ w.r.t. $<$. 
\end{proposition}

\begin{proof}

    We prove the proposition by induction on $k$. Clearly this statement holds for $k=1$ and assume that it holds for $\ell-1$ blocks, next we prove it for $\ell$ blocks. By the inductive assumption we know that $\bigcup_{i=1}^{\ell-1} \pset{G}_i$ is a \grobner basis of $\tilde{I}=\sum_{i=1}^{\ell-1} I_i$. For each $G \in \bigcup_{i=1}^{\ell-1} \pset{G}_i$ and $G_{\ell} \in \pset{G}_{\ell}$, we have $GG_{\ell} \in \tilde{I} \cap I_{\ell}$ and $\lt(GG_{\ell})=\lt(G)\lt(G_{\ell})=\lcm(\lt(G),\lt(G_{\ell}))$, for $\lt(G)$ and $\lt(G_{\ell})$ do not share any variable. Then by Theorem ~\ref{thm:grobnerblcok}, we know that $\bigcup_{i=1}^{\ell} \pset{G}_i$ is a \grobner basis of $\sum_{i=1}^{\ell} I_i$.
\end{proof}

\begin{corollary}\label{cor:disjoint block}

    If $B_i \cap B_j =\emptyset $ for each $i \neq j$, then $\bigcup_{i=1}^k \pset{G}_i$ is a \grobner basis of $I(\p{B}, \p{r})$ w.r.t. $<$. 
\end{corollary}

\begin{proof}
    The proof is straightforward because the variables in $G_i \in \pset{G}_i$ are completely different from those in $G_j \in \pset{G}_j$ when $B_i \cap B_j =\emptyset$ for $i\neq j$. 
\end{proof}

     Let $G$ be an $r_i$-minor in $B_i$ and $B_j$ be some other block with $r_j <r_i$. Then $G$ is said to \emph{attend} $B_j$ if $B_j$ contains at least $r_j$ full rows or full columns of $F$.

\begin{corollary}

    For any $G_i \in \pset{G}_i$ and $G_j\in \pset{G}_j$ with $r_j < r_i$ for $i \neq j$, if either $G_i$ attends $B_j$ or $\lt(G_i)\lt(G_j)=\lcm(\lt(G_i),\lt(G_j))$, then $\bigcup_{i=1}^k \pset{G}_i$ is a \grobner basis of $I(\p{B}, \p{r})$ w.r.t. $<$. 
\end{corollary}

\begin{proof}
    Let $\overline{\pset{G}}_i = \{G \in \pset{G}_i: G \mbox{ does not attend any } B_j \mbox{ for } j \neq i\}.$ 
    For any $G \in \pset{G}_i \setminus \overline{\pset{G}}_i$, assume that it attends $B_j$ with $r_j < r_i$. Then $G$ can be generated by the $r_j$-minors in $B_j$ by applying the generalized Laplace expansion of $G$ w.r.t. $j$ rows or columns, and thus $G$ is redundant as a generator of $I_i$. In other words, $\bigcup_{i=1}^k \overline{\pset{G}}_i$ also generates $I(\p{B}, \p{r})$. By Theorem \ref{thm:grobnerblcok}, one can easily check that $\bigcup_{i=1}^k \overline{\pset{G}}_i$ is a \grobner basis of $I(\p{B}, \p{r})$, and thus $\bigcup_{i=1}^k \pset{G}_i$ is also a \grobner basis of this ideal w.r.t. $<$. 
\end{proof}

\begin{proposition}\label{prop:rowcolumn}

    If for each $G_i \in \pset{G}_i$ and $G_j \in \pset{G}_j$ with $i\neq j$, there exist $x_{cd} \in B_i \cap B_j,x_{ab} \in B_i$ and $x_{\tilde{a}\tilde{b}} \in B_j$ such that the following conditions hold. 
    \begin{enumerate}
        \item[(1)] $R(G_i) \subseteq [c,a]$, $C(G_i) \subseteq [d,b]$, $R(G_j) \subseteq [c,\tilde{a}]$, and $C(G_j) \subseteq [d,\tilde{b}]$;
        \item[(2)] Either \(a\geq \tilde{a},b\leq \tilde{b}\) or \(a\leq \tilde{a},b\geq \tilde{b}\). 
    \end{enumerate}
     Then $\bigcup_{i=1}^k \pset{G}_i$ is a \grobner basis of $I(\p{B}, \p{r})$ w.r.t. $<$. 
\end{proposition}

\begin{proof}
    We prove the proposition by induction on $k$. The statement is obviously true for $k=1$ and assume that it holds for $\ell -1$ blocks, next we prove it for $\ell$ blocks. By the inductive assumption we know that $\bigcup_{i=1}^{\ell-1} \pset{G}_i$ is a \grobner basis of $\tilde{I}=\sum_{i=1}^{\ell-1} I_i$. For any $G \in \bigcup_{i=1}^{\ell-1} \pset{G}_i$ and $G_{\ell} \in \pset{G}_{\ell}$, we have $G \in \pset{G}_i$ for some $i\in [\ell -1]$. For these two minors, by the assumption there exist $x_{cd} \in B_i \cap B_{\ell}$, $x_{ab} \in B_i$, and $x_{\tilde{a}\tilde{b}} \in B_{\ell}$ such that either $a\leq \tilde{a}, b \geq \tilde{b}$ or $a\geq \tilde{a},b\leq \tilde{b}$ holds (Figure~\ref{fortwominors} (left) below illustrates the variables appearing in $\lt(G)$ (marked with $\ast$) and those in $\lt(G_{\ell})$ (marked with orange circles)).

    Since either $a\leq \tilde{a}, b \geq \tilde{b}$ or $a\geq \tilde{a},b\leq \tilde{b}$, we can consider the one-sided ladder with the unique upper corner $x_{cd}$ and two lower corners $x_{ab}$ and $x_{\tilde{a}\tilde{b}}$. We use $\p{x}_{[r_1,r_2],[c_1,c_2]}$ to denote the subset of $\p{x}$ consisting the variables $x_{rc}$ such that $r_1\leq r \leq r_2$ and $c_1\leq c \leq c_2$. Let $\overline{\pset{G}}_i$ and $\overline{\pset{G}}_{\ell}$ be the sets of all the $r_i$-minors in $\p{x}_{[c,a],[d,b]}$ and $r_{\ell}$-minors in $\p{x}_{[c,\tilde{a}],[d,\tilde{b}]}$ respectively. Obviously we have $\overline{\pset{G}}_i \subseteq \pset{G}_i$ and $\overline{\pset{G}}_{\ell} \subseteq \pset{G}_{\ell}$, for $\p{x}_{[c,a],[d,b]} \subseteq B_i$ and $\p{x}_{[c,\tilde{a}],[d,\tilde{b}]}\subseteq B_{\ell}$. From condition~(1) we know that $G \in \overline{\pset{G}}_i$ and $G_{\ell} \in \overline{\pset{G}}_{\ell}$. Let $\Bar{\p{x}}=\p{x}_{[c,a],[d,b]} \cup \p{x}_{[c,\tilde{a}],[d,\tilde{b}]}$ and $\Bar{I}_i$ and $\Bar{I}_{\ell}$ be the ideals generated by $\overline{\pset{G}}_i$ and $\overline{\pset{G}}_{\ell}$ respectively in $\fk[\Bar{\p{x}}]$. Since $\Bar{I}=\Bar{I}_i + \Bar{I}_{\ell}$ is a one-sided ladder determinantal ideal with $\Bar{I}_i$ and $\Bar{I}_{\ell}$ being two homogeneous ideals in $\fk[\Bar{\p{x}}]$, $\overline{\pset{G}}_i \cup \overline{\pset{G}}_{\ell}$ forms a \grobner basis of $\Bar{I}$. By Theorem~\ref{thm:grobnerblcok}, there is a polynomial $H \in \bases{\overline{\pset{G}}_i} \cap \bases{\overline{\pset{G}}_{\ell}}$ such that $\lt(H)=\lcm(\lt(G),\lt(G_{\ell}))$. Then we have $H \in \bases{\pset{G}_i} \cap \bases{\pset{G}_{\ell}} =I_i \cap I_{\ell} \subseteq (\sum_{i=1}^{\ell -1}I_i) \cap I_{\ell}$. By our inductive assumption, $\bigcup_{i=1}^{\ell} \pset{G}_i$ is a \grobner basis of $\sum_{i=1}^{\ell}I_i$ w.r.t. $<$. 
\end{proof}

  \begin{figure}[h]
    \centering
            \begin{subfigure}{.48\textwidth}

        \centering
      \begin{tikzpicture}
        \begin{scope}[scale =0.42]
          \draw[help lines,color=gray!50] (0,0) grid (10,10);
          \draw[color=red] (0,3)--(10,3)--(10,10)--(0,10)--(0,0);
          \draw[color=black] (0,0)--(8,0)--(8,10)--(0,10)--(0,0);

          \draw[orange] (8.5,3.5) circle(0.45);
          \draw[orange] (5.5,4.5) circle(0.45);
          \draw[orange] (0.5,7.5) circle(0.45);
          \draw[orange] (2.5,5.5) circle(0.45);
           
          \node[black] at (6.5,0.5){$\ast$};
          \node[black]  at (4.5,1.5){$\ast$};
          \node[black]  at (3.5,3.5){$\ast$};
          \node[black]  at (2.5,5.5){$\ast$};
          \node[black]  at (1.5,7.5){$\ast$};

         \filldraw[gray!20] (0.05,9.05) rectangle (0.95,9.95);
         \filldraw[gray!20] (7.05,0.05) rectangle (7.95,0.95);
         \filldraw[gray!20] (9.05,3.05) rectangle (9.95,3.95);

          \node[font=\fontsize{7}{7}\selectfont,left, black] at (9.5,0.5){$x_{ab}$};
          \node[font=\fontsize{7}{7}\selectfont,left, black] at (0,9.5){$x_{cd}$};
          \node[font=\fontsize{7}{7}\selectfont,left, black] at (11.5,3.5){$x_{\tilde{a}\tilde{b}}$};

            \end{scope}
          \end{tikzpicture}

        \end{subfigure}
            \begin{subfigure}{.48\textwidth}
      \centering

      \begin{tikzpicture}
           \begin{scope}[scale=1.051]
            \draw[help lines,color=gray!50] (0,0) grid (4,4);
            \draw[color=red] (0,1)--(3,1)--(3,4)--(0,4)--(0,1);
             \draw[color=black] (0,3)--(1,3)--(1,4)--(0,4)--(0,3);

            \foreach \x in {0,1,2}{
                \foreach \y in {1,2,3}{
                  \filldraw[red!20] (\x+0.05,\y+0.05) rectangle (\x+0.95,\y+0.95);
                }        
             }
            \node[black]  at (0.5,3.5){$\ast$};
            
            \end{scope}
          \end{tikzpicture}
                  \end{subfigure}

          \caption{Illustration of $\lt(G)$ and $\lt(G_{\ell})$ in the proof of Proposition~\ref{prop:rowcolumn} (left); Illustration of the two minors in $I_w$ with $w=2143$ in Example~\ref{example:notGB} (right)}\label{fortwominors}
    \label{fig:notGB}
\end{figure}

\begin{example}\label{example:notGB}
    From the proof of Proposition~\ref{prop:rowcolumn}, one may reasonably ask whether the conditions of the proposition can be formulated as: for each $G_i \in \pset{G}_i$ and $G_j \in \pset{G}_j$ with $i\neq j$, there exists a one-sided ladder in $X$ which covers these two minors. But in general this reformulated condition cannot imply that $\bigcup_{i=1}^k \pset{G}_i$ is a \grobner basis of $I(\p{B}, \p{r})$ w.r.t. $<$. Take the Schubert determinantal ideal $I_w$ with $w=2143$ for example, its Futon generators consist of a 1-minor $x_{11}$ in $X_{11}$ and a 3-minor $(\{1,2,3\},\{1,2,3\})$ in $X_{33}$. These two minors is covered by $X_{33}$ which can be identified as a one-sided ladder, but they do not form a \grobner basis of $I_w$ w.r.t. any diagonal term order since $w$ is not vexillary.

\end{example}

\begin{corollary}\label{cor:uniqueuppercorner}
    If the blocks $B_1,\ldots,B_k$ of $I(\p{B}, \p{r})$ are one-sided ladders which share the same unique upper corner, and for any lower corners $x_{a_i b_i}$ of $B_i$ and $x_{\tilde{a}_j \tilde{b}_j}$ of $B_j$ with $i \neq j$, neither $a_i >\tilde{a}_j, b_i > \tilde{b}_j$ nor $a_i < \tilde{a}_j, b_i < \tilde{b}_j$ holds, then $\bigcup_{i=1}^k \pset{G}_i$ is a \grobner basis of $I(\p{B}, \p{r})$ w.r.t. $<$. 
\end{corollary}

\begin{proof}
    Let $x_{cd}$ be the unique upper corner. For each $G_i \in \pset{G}_i$ and $G_j \in \pset{G}_j$ with $i\neq j$, let $x_{a_ib_i}$ and $x_{\tilde{a}_j \tilde{b}_j}$ be any lower corners of $B_i$ and $B_j$ respectively such that $R(G_i) \subseteq [c, a_i], C(G_i) \subseteq [d, b_i]$ and $R(G_j) \subseteq [c, \tilde{a}_i], C(G_j) \subseteq [d, \tilde{b}_i]$ (clearly such lower corners exist). This implies condition~(1) of Proposition~\ref{prop:rowcolumn}, and condition~(2) follows from the assumption of this corollary. 
\end{proof}

  \begin{figure}[h]
    \centering
    \begin{tikzpicture}
       \begin{scope}[scale =0.25]
            \draw[help lines,color=gray!20] (0,1) grid (16,16);
            \draw[color=red] (0,3)--(5,3)--(5,6)--(7,6)--(7,10)--(12,10)--(12,14)--(16,14)--(16,16)--(0,16)--(0,3);
             \draw[color=black] (0,1)--(2,1)--(2,4)--(6,4)--(6,8)--(9,8)--(9,11)--(12,11)--(12,13)--(14,13)--(14,16)--(0,16)--(0,1);
            \draw[densely dashed,color=black](2,1)--(14,1)--(14,13);
            \draw[densely dashed,color=red](5,3)--(16,3)--(16,14);
            
            \filldraw[gray!20] (0.07,15.07) rectangle (0.93,15.93);
            \filldraw[red!20] (6.07,6.07) rectangle (6.93,6.93);
            \filldraw[gray!20] (8.07,8.07) rectangle (8.93,8.93);
            \node[font=\fontsize{6}{6}\selectfont,left, black] at (10.5,6){$x_{a_ib_i}$};
            \node[font=\fontsize{7}{7}\selectfont,left, black] at (0,15.5){$x_{cd}$};
            \node[font=\fontsize{6}{6}\selectfont,left, black] at (12.5,8){$x_{\tilde{a}_j\tilde{b}_j}$};

        \end{scope}
    \end{tikzpicture}
    \caption{Illustration of the block structures in Corollary~\ref{cor:uniqueuppercorner}}
\end{figure}

\section{Reduced \grobner bases of blockwise determinantal ideals via divisibility by leading terms}
\label{sec:reduced}

In this section we first introduce the notion of (anti-)diagonal length of a term in a block and prove that it full characterizes divisibility of this term by leading terms of minors in this block. Then we apply this tool to study the reduced \grobner bases of blockwise determinantal ideals under the assumption that their \grobner bases are known, and then these results are further applied to Schubert and ladder determinantal ideals. 

\subsection{Divisibility by leading terms}

Let $\p{m}$ be an $r$-minor of $X$. Note that $\p{m}$, after expansion, is a homogeneous polynomial in $\kx$ of degree $r$. Then any term $\p{t}$ of $\p{m}$ is indeed the product of $r$ entries of $\p{m}$ from distinct rows and columns of $X$. Write $\p{t} = \prod_{k=1}^r x_{i_kj_k}$ such that $i_1 < i_2 < \cdots < i_r$. Accordingly, we can also use the following representation of a two-row array for $\p{t}$:
\begin{equation}\label{eq:two-row}
  \p{t} = \left(
    \begin{tabular}[c]{cccc}
      $i_1$ & $i_2$ & $\cdots$ & $i_r$ \\
      $j_1$ & $j_2$ & $\cdots$ & $j_r$
    \end{tabular}
    \right).
  \end{equation}
We denote by $f(\p{t})$ and $s(\p{t})$ respectively the first and second sequence of the two-row array representation of $\p{t}$. Note that we restrict $f(\p{t})$ to be strictly increasing for any term $\p{t}$. With this representation, it is clear that a term $\p{t}'$ divides $\p{t}$ if and only if the two-row array of $\p{t}'$ forms a sub-array of that of $\p{t}$.

\begin{lemma}\label{lem:leadingTerm}
  Let $\p{t}$ be a product of $r$ variables from distinct rows and columns in $X$. Then $\p{t}$ is the leading term of some $r$-minor in $X$ w.r.t. the anti-diagonal (resp. diagonal) term order if and only if $s(\p{t})$ is a decreasing (resp. increasing) sequence. 
\end{lemma}

\begin{proof}
  Straightforward. 
\end{proof}

To study reduced \grobner bases we need to test whether a term of some minor in the generators of a blockwise determinantal ideal is divisible by the leading term of some other minor generator, and the term orders of our interest are the anti-diagonal and diagonal ones. For a term $\p{t}$ of some minor in $X$ and a subset $B \subseteq X$, we define the intersection $\p{t} \cap B$ to be the two-row array constructed from that of $\p{t}$ in the form of \eqref{eq:two-row} by removing the $k$th column of $\p{t}$ if $x_{i_k j_k} \not \in B$ for $k \in [r]$.

\begin{definition}\rm \label{def:length}
  Let $\p{t}$ be a term of some minor in $X$ and $B$ be a subset of $X$. Then the \emph{anti-diagonal length} (resp. \emph{diagonal length}) of $\p{t}$ in $B$, denoted by $l_a(\p{t}, B)$ (resp. $l_d(\p{t}, B)$), is the length of the longest decreasing (resp. increasing) subsequence of $s(\p{t} \cap B)$. 
\end{definition}

In case an anti-diagonal or diagonal term order is fixed in the context, we would omit the subscript to write $l(\p{t}, B)$ for the anti-diagonal or diagonal length. When $\p{t}$ is the leading term of some minor generator w.r.t. an anti-diagonal term order, $s(\p{t} \cap B)$ itself is a decreasing sequence, and thus in this case $l_a(\p{t}, B)$ equals the length of $\p{t} \cap B$. Clearly a similar result also holds for any diagonal term order.

\begin{theorem}\label{thm:divisible}
  Let $\p{t}$ be a term of some minor generator from the block $B \in \p{B}$ of an anti-diagonal (resp. diagonal) blockwise determinantal ideal $I(\p{B}, \p{r})$. Then for each $\tilde{B} \in \p{B} \setminus \{B\}$, there exists a minor $\p{m}$ in $\tilde{B}$ such that $\lt_a(\p{m}) | \p{t}$ (resp. $\lt_d(\p{m}) | \p{t}$) if and only if $l_a(\p{t}, \tilde{B}) \geq \tilde{r}$ (resp. $l_d(\p{t}, \tilde{B}) \geq \tilde{r}$), where $\tilde{r}$ is the size of minors in $\tilde{B}$. 
\end{theorem}

\begin{proof} We only prove the theorem for the anti-diagonal term order, and the proof for the diagonal one is exactly the same. 
  
  $\Longleftarrow)$ Let $\p{s}$ be the longest decreasing subsequence of $s(\p{t} \cap \tilde{B})$. Then its length equals $l_a(\p{t}, \tilde{B}) \geq \tilde{r}$. Take the first $\tilde{r}$ elements of $\p{s}$ to form a decreasing sequence $\tilde{\p{s}}$ of length $\tilde{r}$. For each $k=1, \ldots, \tilde{r}$, let $p_k$ be the position of $\tilde{\p{s}}(k)$ in $s(\p{t} \cap \tilde{B})$. Then from the construction of $\p{t} \cap \tilde{B}$, we know that all the variables $x_{f(\p{t} \cap \tilde{B})(p_k), s(\p{t} \cap \tilde{B})(p_k)}$ are in $\tilde{B}$ for $k=1, \ldots, \tilde{r}$. Consider the minor $\p{m}$ with $R(\p{m}) = \{f(\p{t} \cap \tilde{B})(p_k): k \in [\tilde{r}]\}$ and $C(\p{m}) = \{s(\p{t} \cap \tilde{B})(p_k): k \in [\tilde{r}] \}$. Since $x_{f(\p{t} \cap \tilde{B})(p_1), s(\p{t} \cap \tilde{B})(p_1)}$ and $x_{f(\p{t} \cap \tilde{B})(p_{\tilde{r}}), s(\p{t} \cap \tilde{B})(p_{\tilde{r}})}$ are both in the anti-diagonal block $\tilde{B}$, we know that $\p{m}$ is an $\tilde{r}$-minor in $\tilde{B}$ and $\lt(\p{m})$ clearly divides $\p{t}$. 

  $\Longrightarrow)$ Write
  \begin{equation*}
  \lt(\p{m}) = \left(
    \begin{tabular}[c]{cccc}
      $i_1$ & $i_2$ & $\cdots$ & $i_{\tilde{r}}$ \\
      $j_1$ & $j_2$ & $\cdots$ & $j_{\tilde{r}}$
    \end{tabular}
    \right)
  \end{equation*}
  in its two-row array representation. Since $\p{m}$ is a minor in $\tilde{B}$, we know that each $x_{i_k, j_k} \in \tilde{B}$ for $k \in [\tilde{r}]$. With $\lt(\p{m}) | \p{t}$, $x_{i_k, j_k}$ also appears in $\p{t}$, and thus $\lt(\p{m})$ is a sub-array of $\p{t} \cap \tilde{B}$. By Lemma~\ref{lem:leadingTerm} we know that $s(\lt(\p{m}))$ is a decreasing sequence of length $\tilde{r}$, and thus $l_a(\p{t}, \tilde{B}) \geq \tilde{r}$. 
\end{proof}

\begin{example}\rm
  Consider a $5$-minor $\p{m}$ with $R(\p{m}) =  \{5, 7, 8, 10, 12\}$ and $C(\p{m}) = \{2, 4, 6, 8, 10\}$ and its term
  \begin{equation*}
  \p{t} = \left(
    \begin{tabular}[c]{ccccc}
      $5$ & $7$ & $8$ & $10$ & $12$\\
      $2$ & $4$ & $8$ & $10$ & $6$\\
    \end{tabular}
    \right).
  \end{equation*}
  The entries of $\p{m}$ are marked with $*$ and those of $\p{t}$ are in red in Figure~\ref{fig:length} below. Let $B_1$ be a diagonal block illustrated in the left subfigre and the minors in $B_1$ are of size $r_1=3$. One can see that the diagonal length $l_d(\p{t}, B_1) = 3 \geq r_1$ with the entries at $(7,4)$, $(8,8)$, and $(10,10)$ lying in $B_1$. Then the minor $\p{m} = (\{7,8,10\}, \{4, 8, 10\})$ is contained in $B_1$ and satisfies the divisibility $\lt(\p{m}) | \p{t}$.

  However, if we consider the non-diagonal subset $B_2$ in the right subfigure with $r_2 = 3$. Then even with the diagonal length $l_d(\p{t}, B_2) = 3 \geq r_2$ too, the corresponding minor $\p{m}$ does not lie in $B_2$ (with $x_{10,4}$ missing), and thus there is no $3$-minor in $B_2$ whose leading term divides $\p{t}$. This explains why we need to restrict ourselves to (anti-)diagonal subsets to consider blockwise determinantal ideals. 
\end{example}

  \begin{figure}[h]
    \centering
            \begin{subfigure}{.48\textwidth}
      \centering

      \begin{tikzpicture}
        \begin{scope}[scale =0.33]
                \draw[help lines,color=gray!20] (-5,1) grid (12,15);
            \draw[black] (0,1)--(10,1)--(10,8)--(10,12)--(12,12)--(12,15)--
                (2,15)--(2,13)--(-2,13)--(-2,8)--(-5,8)--(-5,1)--(0,1);
                
                \filldraw[red!20] (-3.97,10.03) rectangle (-3.03,10.97);
                \filldraw[red!20] (-1.97,8.03) rectangle (-1.03,8.97);
                \filldraw[red!20] (0.03,3.03) rectangle (0.97,3.97);
                \filldraw[red!20] (2.03,7.03) rectangle (2.97,7.97);
                \filldraw[red!20] (4.03,5.03) rectangle (4.97,5.97);

                \node[black] at (-3.5,3.5){$\ast$};
                \node[black] at (-3.5,5.5){$\ast$};
                \node[black] at (-3.5,7.5){$\ast$};
                \node[black] at (-3.5,8.5){$\ast$};
                \node[black] at (-3.5,10.5){$\ast$};
                \node[black] at (-1.5,3.5){$\ast$};
                \node[black] at (-1.5,5.5){$\ast$};
                \node[black] at (-1.5,7.5){$\ast$};
                \node[black] at (-1.5,8.5){$\ast$};
                \node[black] at (-1.5,10.5){$\ast$};
                \node[black] at (0.5,3.5){$\ast$};
                \node[black] at (0.5,5.5){$\ast$};
                \node[black] at (0.5,7.5){$\ast$};
                \node[black] at (0.5,8.5){$\ast$};
                \node[black] at (0.5,10.5){$\ast$};
                \node[black] at (2.5,3.5){$\ast$};
                \node[black] at (2.5,5.5){$\ast$};
                \node[black] at (2.5,7.5){$\ast$};
                \node[black] at (2.5,8.5){$\ast$};
                \node[black] at (2.5,10.5){$\ast$};
                \node[black] at (4.5,3.5){$\ast$};
                \node[black] at (4.5,5.5){$\ast$};
                \node[black] at (4.5,7.5){$\ast$};
                \node[black] at (4.5,8.5){$\ast$};
                \node[black] at (4.5,10.5){$\ast$};
                
            \end{scope}
          \end{tikzpicture}
        \end{subfigure}
            \begin{subfigure}{.48\textwidth}
      \centering

      \begin{tikzpicture}
           \begin{scope}[scale=0.33]
                \draw[help lines,color=gray!20] (-5,1) grid (12,15);
                \draw[black] (0,1)--(10,1)--(10,8)--(8,8)--(8,12)--(12,12)--(12,15)--
                (2,15)--(2,13)--(-2,13)--(-2,8)--(-5,8)--(-5,6)--(0,6)--(0,1);

                \filldraw[red!20] (-3.97,10.03) rectangle (-3.03,10.97);
                \filldraw[red!20] (-1.97,8.03) rectangle (-1.03,8.97);
                \filldraw[red!20] (0.03,3.03) rectangle (0.97,3.97);
                \filldraw[red!20] (2.03,7.03) rectangle (2.97,7.97);
                \filldraw[red!20] (4.03,5.03) rectangle (4.97,5.97);
                
                \node[black] at (-3.5,3.5){$\ast$};
                \node[black] at (-3.5,5.5){$\ast$};
                \node[black] at (-3.5,7.5){$\ast$};
                \node[black] at (-3.5,8.5){$\ast$};
                \node[black] at (-3.5,10.5){$\ast$};
                \node[black] at (-1.5,3.5){$\ast$};
                \node[black] at (-1.5,5.5){$\ast$};
                \node[black] at (-1.5,7.5){$\ast$};
                \node[black] at (-1.5,8.5){$\ast$};
                \node[black] at (-1.5,10.5){$\ast$};
                \node[black] at (0.5,3.5){$\ast$};
                \node[black] at (0.5,5.5){$\ast$};
                \node[black] at (0.5,7.5){$\ast$};
                \node[black] at (0.5,8.5){$\ast$};
                \node[black] at (0.5,10.5){$\ast$};
                \node[black] at (2.5,3.5){$\ast$};
                \node[black] at (2.5,5.5){$\ast$};
                \node[black] at (2.5,7.5){$\ast$};
                \node[black] at (2.5,8.5){$\ast$};
                \node[black] at (2.5,10.5){$\ast$};
                \node[black] at (4.5,3.5){$\ast$};
                \node[black] at (4.5,5.5){$\ast$};
                \node[black] at (4.5,7.5){$\ast$};
                \node[black] at (4.5,8.5){$\ast$};
                \node[black] at (4.5,10.5){$\ast$};
            
            \end{scope}
          \end{tikzpicture}
                  \end{subfigure}

          \caption{Diagonal length and divisibily by leading terms of minors}\label{fig:length}
\end{figure}
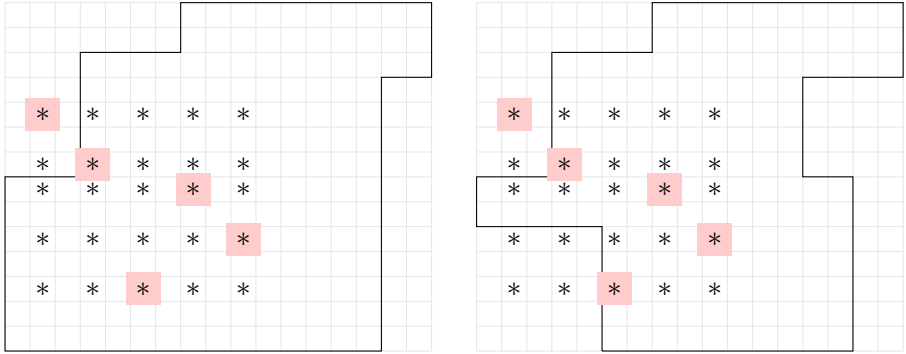

\subsection{Blockwise determinantal ideals}
\label{sec:gb-block}

Let $\p{m}_1$ and $\p{m}_2$ be two distinct minors in $X$. Then we say that $\p{m}_2$ is \emph{contained in} $\p{m}_1$ if $R(\p{m}_2) \subseteq R(\p{m}_1)$ and $C(\p{m}_2) \subseteq C(\p{m}_1)$. It is not hard to have the following criterion for testing reducibility between minors.

\begin{proposition}\label{prop:reducible}
Let $\p{m}_1$ and $\p{m}_2$ be two distinct minors in $X$. Then $\p{m}_1$ is reducible modulo $\p{m}_2$ if and only if $\p{m}_2$ is contained in $\p{m}_1$. 
\end{proposition}

The following results are all presented w.r.t. the diagonal term order for diagonal blockwise determinantal ideals, but indeed they also hold w.r.t. any anti-diagonal term order for anti-diagonal ones. 

\begin{proposition}\label{prop:min-reduced}
Suppose that the generating set $\pset{G}$ of a diagonal blockwise determinantal ideal $I(\p{B}, \p{r})$ is a \grobner basis. For any $r$-minor $\p{m} \in \pset{G}$, let $\p{B}_{\p{m}} = \{B_i \in \p{B}: r_i = \deg(\p{m})\}$. Then the following statements hold:
  \begin{enumerate}
  \item[(1)] If for each $\p{m} \in \pset{G}$, $l(\lt(\p{m}), \tilde{B}) < \tilde{r}$ for any $\tilde{B} \in \p{B} \setminus \p{B}_{\p{m}}$, where $\tilde{r}$ is the size of minors in $\tilde{B}$, then $\pset{G}$ is a minimal \grobner basis of $I(\p{B}, \p{r})$. 
  \item[(2)] If for each $\p{m} \in \pset{G}$ and for any term $\p{t}$ of $\p{m}$, we have $l(\p{t}, \tilde{B}) < \tilde{r}$ for any $\tilde{B} \in \p{B} \setminus \p{B}_{\p{m}}$, then $\pset{G}$ is the reduced \grobner basis of $I(\p{B}, \p{r})$. 
  \end{enumerate}
\end{proposition}

\begin{proof}

  Let $\p{m}$ be an $r$-minor and $\tilde{\p{m}}$ be an arbitrary $r$-minor in some $B \in \p{B}_{\p{m}}$ distinct from $\p{m}$. Then for each term $\p{t}$ of $\p{m}$, we claim that $\lt(\tilde{\p{m}})$ and $\p{t}$ are distinct. This is because otherwise if $\lt(\tilde{\p{m}}) = \p{t}$, then $s(\p{t})$ is a decreasing sequence of length $r$, and thus by Lemma~\ref{lem:leadingTerm} we have $\p{t} = \lt(\p{m})$. Since different $r$-minors have different leading terms, we have $\lt(\tilde{\p{m}}) \neq \lt(\p{m})$, and thus $\lt(\tilde{\p{m}})$ does not divide any term $\p{t}$ of $\p{m}$.

  For the remaining case when $\tilde{\p{m}}$ is a minor from other $\tilde{B} \in \p{B} \setminus \p{B}_{\p{m}}$, by Theorem~\ref{thm:divisible} we know that in (1) $\lt(\tilde{\p{m}})$ does not divide $\lt(\p{m})$ and that in (2) $\lt(\tilde{\p{m}})$ does not divide any term $\p{t}$ of $\p{m}$ either. To summarize, $\pset{G}$ is a minimal \grobner basis of $I(\p{B}, \p{r})$ in (1) and is the reduced one in (2).
\end{proof}

Let $\p{m}$ be a minor generator from $B \in \p{B}$ and $\tilde{B} \neq B$. We use $\p{m} \cap \tilde{B}$ to denote the intersection of the submatrix corresponding to $\p{m}$ and the block $\tilde{B}$ and $R(\p{m} \cap \tilde{B}) := \{i \in [m]: \exists x_{ij} \in \p{m} \cap \tilde{B}\}$. $C(\p{m} \cap \tilde{B})$ can be similarly defined. Clearly for any term $\p{t}$ of $\p{m}$, we have $l(\p{t}, \tilde{B}) \leq \#R(\p{m} \cap \tilde{B})$ and $l(\p{t}, \tilde{B}) \leq \#C(\p{m} \cap \tilde{B})$. That is, the length of any term of $\p{m}$ in $\tilde{B}$ is bounded by the number of rows (or columns) of the submatrix of $\p{m}$ appearing in $\tilde{B}$. Note that a row (or column) appearing in $\tilde{B}$ does not necessarily means that the corresponding full row (or column) belongs to $\tilde{B}$. With these straightforward inequalities, we have the following sufficient conditions for a \grobner basis of a diagonal blockwise determinantal ideal to be the reduced one.

\begin{corollary}\label{cor:fewerRows}
Suppose that the generating set $\pset{G}$ of a diagonal blockwise determinantal ideal $I(\p{B}, \p{r})$ is a \grobner basis. If for each $\p{m} \in \pset{G}$, $\#R(\p{m} \cap \tilde{B}) < \tilde{r}$ or $\#C(\p{m} \cap \tilde{B}) < \tilde{r}$ for any $\tilde{B} \in \p{B} \setminus \p{B}_{\p{m}}$, where $\p{B}_{\p{m}} = \{B_i \in \p{B}: r_i = \deg(\p{m})\}$, then $\pset{G}$ is the reduced \grobner basis of $I(\p{B}, \p{r})$.
\end{corollary}

Next we consider the special case when a minor $\p{m}$ in $B$ intersects another block $\tilde{B} \neq B$ with consecutive full rows or columns of $\p{m}$. 

\begin{proposition}\label{prop:fullRows}
  Let $\p{m}$ be a minor generator of a diagonal blockwise determinantal ideal $I(\p{B}, \p{r})$ such that $\p{m} \cap \tilde{B}$ consists of consecutive full rows or columns of the submatrix of $\p{m}$ for any $\tilde{B} \in \p{B} \setminus \p{B}_{\p{m}}$, where $\p{B}_{\p{m}} = \{B_i \in \p{B}: r_i = \deg(\p{m})\}$. Then for any $\tilde{\p{m}}$ in $\tilde{B}$, if $\lt(\tilde{\p{m}})\nmid \lt(\p{m})$, then $\p{m}$ is reduced modulo $\tilde{\p{m}}$, namely $\lt(\tilde{\p{m}}) \nmid  \p{t}$ for each term $\p{t}$ of $\p{m}$. 
\end{proposition}

\begin{proof}
  Suppose that $\p{m} \cap \tilde{B}$ consists of $k$ consecutive full rows or columns of the submatrix of $\p{m}$. Then clearly we have $l(\lt(\p{m}), \tilde{B}) = k$ and $l(\p{t}, \tilde{B}) \leq k$ for any term $\p{t}$ of $\p{m}$. Then with Proposition~\ref{prop:min-reduced} the conclusion follows. 
\end{proof}

\begin{corollary}\label{cor:fullRows}
  Suppose that the generating set $\pset{G}$ of a diagonal blockwise determinantal ideal $I(\p{B}, \p{r})$ is a minimal \grobner basis. If for each $\p{m} \in \pset{G}$, $\p{m} \cap \tilde{B}$ consists of consecutive full rows or columns of the submatrix of $\p{m}$ for any $\tilde{B} \in \p{B} \setminus \p{B}_{\p{m}}$, where $\p{B}_{\p{m}} = \{B_i \in \p{B}: r_i = \deg(\p{m})\}$, then $\pset{G}$ is the reduced \grobner basis of $I(\p{B}, \p{r})$.
\end{corollary}

\subsection{Schubert determinantal ideals}
 \label{sec:gb-schubert}

We found the same results on reduced \grobner bases of Schubert determinantal ideals as in this section already presented in \cite[Theorem~1.5]{S23m} when we were finishing this manuscript. We decided to keep this section for the reason of completeness, and the readers should also refer to \cite{S23m} for these results. 

 Note that the defining blocks $X_{pq}$ with $(p,q) \in \ess(w)$ of a Schubert determinantal ideal $I_w$ are both anti-diagonal and diagonal. So we can study the reduced \grobner basis of $I_w$ w.r.t.\! both anti-diagonal and diagonal term orders with the results presented in Section~\ref{sec:gb-block}. 

Let $w \in S_n$ be a permutation. For any two distinct $(p, q), (\tilde{p}, \tilde{q}) \in \ess(w)$, let $r= \rank(w^T_{pq})+1$ and $B = X_{pq}$, and $\tilde{r}$ and $\tilde{B}$ be similarly defined for $(\tilde{p}, \tilde{q})$. To study the reduced \grobner basis of $I_w$, we need to test whether an $r$-minor in $B$ is reducible modulo any $\tilde{r}$-minor in $\tilde{B}$. Clearly we only need to consider the case when $r> \tilde{r}$. There exist two kinds of relative positions between $(p, q)$ and $(\tilde{p}, \tilde{q})$: (1) $\tilde{p} < p$ and $\tilde{q} < q$, and (2) otherwise, namely $\tilde{p} \geq p$ or $\tilde{q} \geq q$.

Case~(2) is easier: in this case for any $r$-minor $\p{m}$ in $B$, $\p{m} \cap \tilde{B}$ is either empty or consists of consecutive full rows or columns of the submatrix of $\p{m}$. Then by Proposition~\ref{prop:fullRows}, we only need to ensure that $\lt(\tilde{\p{m}}) \nmid \lt(\p{m})$ for any $\tilde{\p{m}}$ in $\tilde{B}$, and this is guaranteed by the minimality of elusive minors in Theorem~\ref{thm:minimal}. 

Case~(1) only happens when $w$ is not vexillary, and one can see that this case contains the following two sub-cases: (1.1) $\p{m} \cap \tilde{B}$ is either empty or consists of consecutive rows or columns of the submatrix of $\p{m}$, this case is similar to case (2) above; (1.2) $\p{m} \cap \tilde{B}$ does not contain full rows or columns of the submatrix of $\p{m}$, namely $1 \leq \#R(\p{m} \cap \tilde{B}) < r$ and $1 \leq \#C(\p{m} \cap \tilde{B}) < r$. In this sub-case, for any anti-diagonal term order, one can see that there exists a term $\p{t}$ of $\p{m}$ other than $\lt(\p{m})$ such that $l(\p{t}, \tilde{B}) > l(\lt(\p{m}), \tilde{B})$. Therefore, unlike case~(2) above, $l(\lt(\p{m}), \tilde{B}) < \tilde{r}$ does not imply $l(\p{t}, \tilde{B}) < \tilde{r}$. Take $w = 2143 \in S_4$ for example. Then $\ess(w) = \{(1, 1), (3,3)\}$, and the leading term of the only $3$-minor $\p{m}$ in $X_{33}$ is $x_{13}x_{22}x_{31}$, not divisible by the only $1$-minor $x_{11}$ in $X_{11}$. However, two terms of $\p{m}$, $x_{11}x_{23}x_{32}$ and $x_{11}x_{22}x_{33}$, are divisible by $x_{11}$. Hence, Fulton generators, which are also the elusive minors for this $w$, form a minimal \grobner basis of $I_w$ but is not the reduced one. In fact, in this sub-case reducibility between elusive minors always occurs w.r.t. the anti-diagonal term order.

\begin{theorem}\label{thm:schubert-vex-reduced}
  Let $w \in S_n$ be a vexillary permutation. Then elusive minors form the reduced \grobner basis of $I_w$ w.r.t. any anti-diagonal or diagonal term order. 
\end{theorem}

\begin{proof}
   If $w$ is vexillary, then for any two distinct $(p, q), (\tilde{p}, \tilde{q}) \in \ess(w)$, we have $\tilde{p} \geq p$ or $\tilde{q} \geq q$. Hence for any minor $\p{m}$ in $B$, $\p{m} \cap \tilde{B}$ is either empty or consists of consecutive rows or columns of the submatrix of $\p{m}$. By Theorem~\ref{thm:minimal}, the set of elusive minors form a minimal \grobner basis of $I_w$, and thus further by Corollary~\ref{cor:fullRows} we know that they are the reduced \grobner basis.
\end{proof}

\begin{theorem}\label{thm:schubert-reduced-vex}
Let $w \in S_n$ be a permutation. If elusive minors form the reduced \grobner basis of $I_w$ w.r.t. an anti-diagonal term order, then $w$ is vexillary. 
\end{theorem}

\begin{proof}
  We prove that if $w$ is not vexillary, then there exists an elusive minor which is reducible modulo another elusive minor. Since $w$ is not vexillary, there exist two distinct $(p, q), (\tilde{p}, \tilde{q}) \in \ess(w)$ with $\tilde{p} < p$ and $\tilde{q} < q$. Consider the submatrix $M$ of $w^{T}$ with the entries at $(\tilde{p}+1, \tilde{q}+1)$ and $(p, q)$ as its diagonal corners. Let $(\overline{p}, \overline{q})$ be one of the northwest most elements in $M$ such that $(\overline{p}, \overline{q})$ belongs to the Rothe diagram $D(w)$ and any other element in $\{(i, j): i \in [\tilde{p}+1,\overline{p}], j \in [\tilde{q}+1, \overline{q}]\}$ does not. With $(p, q) \in \ess(w) \subseteq D(w)$ and $(p, q) \in M$, clearly such an element $(\overline{p}, \overline{q})$ exists. The relative positions of $(p, q)$, $(\tilde{p}, \tilde{q})$, and $(\overline{p}, \overline{q})$ are illustrated in Figure~\ref{fig:schubert-vex}. Denote $\overline{r} = \rank(w^T_{\overline{p}\overline{q}})+1$ and $\tilde{r} = \rank(w^T_{\tilde{p}\tilde{q}})+1$. Let $\p{m} = ([\overline{p}-\overline{r}+1, \overline{p}], [\overline{q}-\overline{r}+1, \overline{q}])$. Then by Proposition~\ref{prop:corner}, $\p{m}$ is an elusive minor of size $\overline{r}$, and next we prove that $\p{m}$ is reducible modulo some elusive minor in $X_{\tilde{p}\tilde{q}}$.

By the assumption, none of $(\overline{p}, \tilde{q}+1), \ldots, (\overline{p}, \overline{q}-1)$ and $(\tilde{p}+1, \overline{q}), \ldots, (\overline{p}-1, \overline{q})$ belongs to $D(w)$. But with $(\overline{p}, \overline{q}) \in D(w)$, we know that any of $(\overline{p}, \tilde{q}+1), \ldots, (\overline{p}, \overline{q}-1)$ can only be removed from $D(w)$ by some $(k, w(k))$ to its north, and similarly any of $(\tilde{p}+1, \overline{q}), \ldots, (\overline{p}-1, \overline{q})$ only by some $(k, w(k))$ to its west. Furthermore, with $(\tilde{p}, \tilde{q}) \in \ess(w)$ we know that neither $(\tilde{p}, \tilde{q}+1)$ nor $(\tilde{p}+1, \tilde{q})$ lies in $D(w)$. Similarly, $(\tilde{p}+1, \tilde{q})$ can only be removed by some $(k, w(k))$ to its west and $(\tilde{p}, \tilde{q}+1)$ only by some $(k, w(k))$ to its north. All the boxes outside $D(w)$ we use are marked in red in Figure~\ref{fig:schubert-vex}.

Recall that there are $\overline{r}$ and $\tilde{r}$ non-zero entries in $w^{T}_{\overline{p}\overline{q}}$ and $w^{T}_{\tilde{p}\tilde{q}}$ respectively. Counting the entries $(k, w(k))$ in $w^{T}_{\overline{p}\overline{q}}$ but not in $w^{T}_{\tilde{p}\tilde{q}}$ to remove $(\overline{p}, \tilde{q}+1), \ldots, (\overline{p}, \overline{q}-1)$ and $(\tilde{p}+1, \tilde{q})$, we have $\overline{r} - \tilde{r} \geq (\overline{q}-\tilde{q}-1)+ 1 = \overline{q} - \tilde{q}$. Similarly we also have $\overline{r} - \tilde{r} \geq \overline{p} - \tilde{p}$. These mean that $\#(R(\p{m}) \cap [\tilde{p}]) = \tilde{p} - (\overline{p}-\overline{r}+1)+1 =  \overline{r} - (\overline{p} - \tilde{p}) \geq \tilde{r}$ and $\#(C(\p{m}) \cap [\tilde{q}]) = \tilde{q} - (\overline{q}-\overline{r}+1)+1 = \overline{r} - (\overline{q} - \tilde{q}) \geq \tilde{r}$. Consider the minor $\tilde{\p{m}} = ([\tilde{p}-\tilde{r}+1, \tilde{p}], [\tilde{q}-\tilde{r}+1, \tilde{q}])$. Then again by Proposition~\ref{prop:corner} we know that $\tilde{\p{m}}$ is an elusive minor of size $\tilde{r}$ in $X_{\tilde{p}\tilde{q}}$. Clearly $R(\tilde{\p{m}}) \subseteq R(\p{m})$ and $C(\tilde{\p{m}}) \subseteq C(\p{m})$, and thus by Proposition~\ref{prop:reducible} we know that $\p{m}$ is reducible modulo $\tilde{\p{m}}$: this ends the proof. 
\end{proof}

  \begin{figure}[h]
    \centering

\begin{tikzpicture}
    \begin{scope}[scale =0.5]
        \draw[black] (0,0)--(10,0)--(10,10)--(0,10)--(0,0);
        \draw[black] (0,6)--(3,6)--(3,10);
        \draw[black] (0,2)--(8,2)--(8,10);
        \filldraw[gray!20] (2.03,6.03) rectangle (2.97,6.97);
        \filldraw[gray!20] (7.03,2.03) rectangle (7.97,2.97);
        \filldraw[gray!20] (9.03,0.03) rectangle (9.97,0.97);
        \node[font=\fontsize{6}{6}\selectfont,left,black] at (2.2,6.5){$(\tilde{p},\tilde{q})$};
        \node[font=\fontsize{6}{6}\selectfont,right,black] at (7.8,2.5){$(\bar{p},\bar{q})$};
        \node[font=\fontsize{6}{6}\selectfont,right,black] at (9.8,0.5){$(p,q)$};

        \foreach \x in {3,...,6}{
            \filldraw[red!20] (\x,2.03) rectangle (\x+1,2.97);}
        \foreach \x in {3,...,5}{
            \filldraw[red!20] (7.03,\x) rectangle (7.97,\x+1);}
    
        \foreach \x in {3,...,7}{
          \foreach \y in {2,...,5}{
          \draw[black] (\x,\y) rectangle (\x+1,\y+1);}}
        \draw[black] (2,5) rectangle (3,6);
        \draw[black] (3,6) rectangle (4,7);
        \filldraw[red!20] (2.03,5.03) rectangle (2.97,5.97);
        \filldraw[red!20] (3.03,6.03) rectangle (3.97,6.97);
    \end{scope}
  \end{tikzpicture}
  \caption{Illustration for the proof of Theorem~\ref{thm:schubert-reduced-vex}}\label{fig:schubert-vex}
  \end{figure}
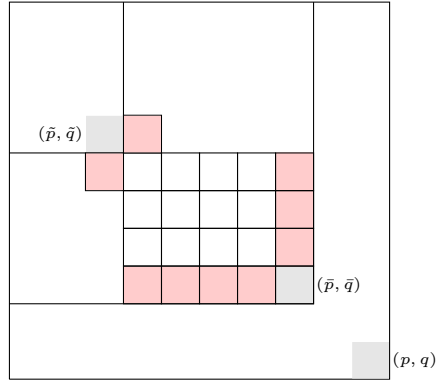

As one can see, the reduced \grobner bases of $I_w$ for a non-vexillary permutation $w$ w.r.t. anti-diagonal term orders are still unknown. This remaining case is further studied in Section~\ref{sec:reduction} by investigating the reduction among elusive minors.

\subsection{Ladder determinantal ideals}
\label{sec:gb-ladder}

By Theorem~\ref{thm:one-sided-schubert} one-sided ladder determinantal ideals can be identified as Schubert determinantal ideals of vexillary permutations. For any one-sided ladder determinantal ideal $I$, let $w$ be its corresponding vexillary permutation in Theorem~\ref{thm:one-sided-schubert}. Then by Theorem~\ref{thm:schubert-vex-reduced}, we know that all the elusive minors form the reduced \grobner basis of $I$ w.r.t. any anti-diagonal and diagonal term order. 

For two-sided ladder determinantal ideals, the \grobner basis in Theorem~\ref{thm:mixed} is indeed the reduced one, as shown below. 

\begin{proposition}\label{prop:mixed-min}
  Let $I(L,\p{r})$ and $\pset{G}$ be as stated in Theorem~\ref{thm:mixed}. Then $\pset{G}$ is the reduced \grobner basis of $I(L,\p{r})$  w.r.t. any lexicographic diagonal term order induced by a scanning variable order.
\end{proposition}

\begin{proof}
This proposition follows immediately from Corollary~\ref{cor:fewerRows}.
\end{proof}

\section{Reduction between minors}
\label{sec:reduction}

As explained in Section~\ref{sec:gb}, reduction is the fundamental operation in the theory of \grobner bases, and from a minimal \grobner basis one can construct the reduced \grobner basis by performing reduction. In this section we are interested in reduction of a minor in $X$ modulo another set of minors, and the term order is an anti-diagonal or diagonal one.

\begin{definition}\label{def:remove}
  Let $\p{m}_1$ and $\p{m}_2$ be two distinct minors in $X$ such that $\p{m}_2$ is contained in $\p{m}_1$. The \emph{complement} of $\p{m}_2$ in $\p{m}_1$, denoted by $\overline{\p{m}_2}$, is the minor such that $R(\overline{\p{m}_2}) = R(\p{m}_1) \setminus R(\p{m}_2)$  and $C(\overline{\p{m}_2}) = C(\p{m}_1) \setminus C(\p{m}_2)$. The \emph{removed terms} of $\p{m}_1$ by $\p{m}_2$ is defined as
$$\rterm(\p{m}_1, \p{m}_2) := \{\p{t}\overline{\p{t}} : \p{t}\in \term(\p{m}_2), \overline{\p{t}}\in \term(\overline{\p{m}_2})\}.$$
\end{definition}

Clearly we have $\rterm(\p{m}_1, \p{m}_2) \subseteq \term(\p{m}_1)$. The set of removed terms of $\p{m}_1$ by $\p{m}_2$ is indeed the terms of the \emph{Laplace product} of $\p{m}_2$ in $\p{m}_1$ in \cite{S16o}. The word ``Laplace'' here is used because it appears as a summand in the generalized Laplace expansion of a square matrix $X$. That is, the determinant of $X$ is equal to all the Laplace products which share the same row indices $R$: $\det(X) = \sum_{R(\p{m}) = R} (-1)^{\delta(\p{m})} \p{m}\overline{\p{m}}$, where $\delta(\p{m})$ is the sum of all the elements in $R(\p{m})$ and $C(\p{m})$ \cite{S16o}. The following proposition justifies the name of removed terms in Definition~\ref{def:remove}.

\begin{proposition}\label{prop:reduction}
 Let $\p{m}_1$ and $\p{m}_2$ be two distinct minors in $X$ such that $\p{m}_1$ is reducible modulo $\p{m}_2$. Write $\p{m}_1 = \sum_{\p{t} \in \term(\p{m}_1)}c_{\p{t}}\p{t}$ with $c_{\p{t}} = \pm 1$. Then
  \begin{equation}
    \label{eq:reduction}
   \p{m}_1 \xrightarrow[*]{\p{m}_2} \sum_{\p{t} \in \term(\p{m}_1)\setminus \rterm(\p{m}_1,\p{m}_2)}c_{\p{t}}\p{t}.  
  \end{equation}  
\end{proposition}

\begin{proof}

  On one hand, for any term $\p{t} \in \term(\p{m}_1)$ such that $\lt(\p{m}_2) | \p{t}$, let $\overline{\p{t}} = \p{t} / \lt(\p{m}_2)$. By Proposition~\ref{prop:reducible} we have $\overline{\p{t}} \in \term(\overline{\p{m}_2})$. On the other hand, for any term $\overline{\p{t}} \in \term(\overline{\p{m}_2})$, $\p{t} = \lt(\p{m}_2) \overline{\p{t}}$ is a term of $\p{m}_1$ and $\lt(\p{m}_2) | \p{t}$. In other words, there exists a one-one correspondence between terms of $\p{m}_1$ divisible by $\lt(\p{m}_2)$ and the terms of $\overline{\p{m}_2}$.

  Furthermore, for each $\overline{\p{t}} \in \term(\overline{\p{m}_2})$, we know that $\term(\overline{\p{t}} \p{m}_2) \subseteq \term(\p{m}_1)$ and the coefficients of the terms in $\term(\overline{\p{t}} \p{m}_2)$ are either identical to those of the corresponding terms in $\term(\p{m}_1)$ or equal to their minuses. From this we have $ \p{m}_1 \xrightarrow[\overline{\p{t}}\lt(\p{m}_2)]{\p{m}_2}\p{m}_{1, \overline{\p{t}}}$, where $\p{m}_{1, \overline{\p{t}}}$ is obtained by removing the terms in $\term(\overline{\p{t}} \p{m}_2)$ from $\p{m}_1$. Note that for distinct $\overline{\p{t}}_1, \overline{\p{t}}_2 \in \term(\overline{\p{m}_2})$, we have $\term(\overline{\p{t}}_1 \p{m}_2) \cap \term(\overline{\p{t}}_2 \p{m}_2) = \emptyset$. Then by the definition of $\rterm(\p{m}_1,\p{m}_2)$, equation~\eqref{eq:reduction} follows. 
\end{proof}

\begin{example}\rm
  Take $\p{m}_1 = (\{1, 2, 3, 4\}, \{1, 2, 3, 4\})$ and $\p{m}_2 = (\{2, 3\}, \{1, 2\})$ for example. Since $R(\p{m}_2) \subseteq R(\p{m}_1)$ and $C(\p{m}_2) \subseteq C(\p{m}_1)$, we know that $\p{m}_1$ is reducible modulo $\p{m}_2$. Next let us consider the reduction of $\p{m}_1$ module $\p{m}_2$ for an anti-diagonal term order. Note that $\p{m}_2 = - x_{22}x_{31} + x_{21}x_{32}$ and $\lt(\p{m}_2) = x_{22}x_{31}$. First $\p{t}_1 = x_{13}x_{22}x_{31}x_{44} \in \term(\p{m}_1)$ is divisible by $\lt(\p{m}_2)$, and removing $\p{t}_1$ from $\p{m}_1$ modulo $\p{m}_2$ results in $\p{m}' := \p{m}_1 + x_{13}x_{44}\p{m}_2$. One can easily check that $\p{m}'$ is obtained by removing two terms $x_{13}x_{22}x_{31}x_{44}$ and $x_{13}x_{21}x_{33}x_{44}$ from $\p{m}_1$. Then $\p{t}_2 = x_{14}x_{22}x_{31}x_{43} \in \term(\p{m}')$ is divisible by $\lt(\p{m}_2)$, and the reduction results in two additional terms $x_{14}x_{22}x_{31}x_{43}$ and $x_{13}x_{22}x_{31}x_{44}$ removed from $\p{m}'$, leaving no term divisible by $\lt(\p{m}_2)$. As one can check, the removed four terms are precisely those described in $\rterm(\p{m}_1, \p{m}_2)$. 
\end{example}

Furthermore, let $\pset{M} = \{\p{m}_1, \ldots, \p{m}_s\}$ be a set of minors such that $\p{m}$ is reducible modulo $\p{m}_i$ for each $i=1, \ldots, s$. Then the \emph{removed terms} of $\p{m}$ by $\pset{M}$ is defined as $\rterm(\p{m},\pset{M}) := \bigcup_{i=1}^s \rterm(\p{m}, \p{m}_i)$. Our goal in this section is to extend formula~\eqref{eq:reduction} in Proposition~\ref{prop:reduction} to that for reduction of $\p{m}$ modulo $\pset{M}$ with $\rterm(\p{m},\pset{M})$. As one can easily find from the simple example with $\p{m}=(\{1,2,3\}, \{1,2,3\})$ and $\pset{M}=\{(\{1,2\}, \{1,2\}), (\{2,3\}, \{2,3\})\}$, in general this extension does not work. That is to say, certain structures of the minors in $\pset{M}$ are necessary for this extension. 

Next let us work on the Schubert determinantal ideal $I_w$ for some permutation $w \in S_n$. Since elusive minors of $I_w$ form a minimal \grobner basis by Theorem~\ref{thm:minimal} and they are also the reduced one by Theorem~\ref{thm:schubert-vex-reduced} for vexillary permutations, it remains to study the reduction among elusive minors to transform them to the reduced \grobner basis of $I_w$ for a non-vexillary permutation $w$ w.r.t. the anti-diagonal term order. The remaining part of this section is devoted to proving the following theorem extending Proposition~\ref{prop:reduction}.

\begin{theorem}\label{thm:reductionM}
  Let $w \in S_n$ be a non-vexillary permutation and consider the anti-diagonal term order. For any elusive minor $\p{m}$ of $I_w$, let $\pset{E}_{\p{m}}$ be the set of elusive minors strictly contained in $\p{m}$. Write $\p{m} = \sum_{\p{t} \in \term(\p{m})}c_{\p{t}}\p{t}$ with $c_{\p{t}} = \pm 1$. Then
  \begin{equation}
    \label{eq:reductionM}
   \p{m}\xrightarrow[*]{\pset{E}_{\p{m}}} \sum_{\p{t} \in \term(\p{m})\setminus \rterm(\p{m},\pset{E}_{\p{m}})}c_{\p{t}}\p{t}.  
  \end{equation} 
\end{theorem}

Structurally, Fulton generators are easier to use compared with elusive minors. So our overall strategy to prove the theorem above is to replace $\pset{E}_{\p{m}}$ with the set $\pset{F}_{\p{m}}$ of all Fulton generators strictly contained in $\p{m}$, prove equation~\eqref{eq:reductionM} for $\pset{F}_{\p{m}}$ and the equality of the removed sets $\rterm(\p{m},\pset{F}_{\p{m}}) = \rterm(\p{m},\pset{E}_{\p{m}})$. We will also prove the equality $\lt(\pset{F}_{\p{m}}) = \lt(\pset{E}_{\p{m}})$, and thus the replacement above still yields to reduction of $\p{m}$ towards the reduced \grobner basis.

We assign a total order $<_r$ to all the Fulton generators $\pset{F}$ of $I_w$ as follows. For any distinct $\p{m}_1, \p{m}_2 \in \pset{F}$, we say that $\p{m}_1 <_r \p{m}_2$ if one of the following conditions holds, where $R(\p{m}_1)_i$ denotes the $i$-th smallest element in $R(\p{m}_1)$ and $C(\p{m}_1)_i$ is similarly defined:
\begin{itemize}
\item[(1)] $\deg(\p{m}_1) < \deg(\p{m}_2)$; 
\item[(2)] $\deg(\p{m}_1) = \deg(\p{m}_2)$, and there exists an integer $k \in [\deg(\p{m}_1)]$ such that $R(\p{m}_1)_i = R(\p{m}_2)_i$ for $i \in [k-1]$ and $R(\p{m}_1)_k > R(\p{m}_2)_k$;
\item[(3)] $R(\p{m}_1) = R(\p{m}_2)$, and there exists an integer $k \in [\deg(\p{m}_1)]$ such that $C(\p{m}_1)_i = C(\p{m}_2)_i$ for $i \in [k-1]$ and $C(\p{m}_1)_k > C(\p{m}_2)_k$.
\end{itemize}

One can easily check that $<_r$ is a total order on $\pset{F}$, and we will use $<_r$ to order the Fulton generators in $\pset{F}_{\p{m}}$ to reduce $\p{m}$. To be specific, for a given elusive minor $\p{m}$ of $I_w$, let $\pset{F}_{\p{m}} = \{\p{f}_1, \ldots, \p{f}_s\}$ be the set of Fulton generators strictly contained in $\p{m}$ such that $\p{f}_1 >_r \cdots >_r \p{f}_s$. Then we reduce $\p{m}$ modulo $\pset{F}_{\p{m}}$ in the following order
$$   \p{m}\xrightarrow[*]{\p{f}_1} \p{m}_1  \xrightarrow[*]{\p{f}_2} \cdots \xrightarrow[*]{\p{f}_s} \p{m}_s,$$
and prove that $\p{m}_s$ is the reduction result of $\p{m}$ modulo $\pset{F}_{\p{m}}$ and it has the form stated in equation~\eqref{eq:reductionM}. To prove this we need some preparations. 

\begin{lemma}\label{lem:smaller}
  Let $\p{m}$ be a minor containing two minors $\p{f}_1$ and $\p{f}_2$. If $R(\p{f}_2) \subseteq R(\p{f}_1)$ and $C(\p{f}_2) \not \subseteq C(\p{f}_1)$ or $C(\p{f}_2) \subseteq C(\p{f}_1)$ and $R(\p{f}_2) \not \subseteq R(\p{f}_1)$, then $\rterm(\p{m}, \p{f}_1) \cap \rterm(\p{m}, \p{f}_2) = \emptyset$. 
\end{lemma}

\begin{proof}
  We only prove the former case when $R(\p{f}_2) \subseteq R(\p{f}_1)$ and $C(\p{f}_2) \not \subseteq C(\p{f}_1)$, and the proof for the latter is similar. Since $C(\p{f}_2) \not \subseteq C(\p{f}_1)$, there exists an integer $j \in C(\p{f}_2) \setminus C(\p{f}_1)$. Then for any $i \in R(\p{f}_2) \subseteq R(\p{f}_1)$, $x_{ij}$ is in $R(\p{f}_2) \times C(\p{f}_2)$ but in neither $R(\p{f}_1) \times C(\p{f}_1)$ nor $R(\overline{\p{f}_1}) \times C(\overline{\p{f}_1})$. This implies that any term in $\rterm(\p{m}, \p{f}_1)$ does not involve $x_{ij}$, and by the arbitrariness of $i \in R(\p{f}_2)$ the conclusion follows. 
\end{proof}

\begin{corollary}\label{cor:sameRow}
  Let $\p{m}$ be a minor containing two distinct minors $\p{f}_1$ and $\p{f}_2$ of the same size. If $R(\p{f}_1) = R(\p{f}_2)$ or $C(\p{f}_1) = C(\p{f}_2)$, then $\rterm(\p{m}, \p{f}_1) \cap \rterm(\p{m}, \p{f}_2) = \emptyset$. 
\end{corollary}

Recall that for $\p{f}$ contained in $\p{m}$, the removed terms of $\p{m}$ by $\p{f}$ is $\rterm(\p{m}, \p{f}) = \{\tilde{\p{t}} \overline{\p{t}} : \tilde{\p{t}}\in \term(\p{f}), \overline{\p{t}} \in \term(\overline{\p{f}})\} = \bigcup_{\overline{\p{t}}\in \term(\overline{\p{f}})} \{\tilde{\p{t}} \overline{\p{t}} : \tilde{\p{t}}\in \term(\p{f})\}$. For a fixed $\overline{\p{t}}\in \term(\overline{\p{f}})$, the subset $\{\tilde{\p{t}} \overline{\p{t}} : \tilde{\p{t}} \in \term(\p{f})\}$ is called the $\overline{\p{t}}$-\emph{section} of $\p{f}$ in $\p{m}$, denoted by $\p{f}_{\overline{\p{t}}}^{\p{m}}$ and by $\p{f}_{\overline{\p{t}}}$ for simplicity if $\p{m}$ is known from the context. In the following we study the relationships between sections of a minor $\p{f}$ in $\p{m}$ and the removed set of $\p{m}$ by the minors greater than $\p{f}$ w.r.t. $<_r$ for Schubert determinantal ideals. We first investigate the case of minors of the same size as $\p{f}$ and then that of minors of larger sizes than $\p{f}$.

\begin{proposition}\label{prop:sameSize}
  Let $w \in S_n$ and $\p{m}$ be an elusive minor of $I_w$ reducible by a Fulton generator $\p{f}$. Suppose that there exists a term $\p{t} \in \rterm(\p{m}, \p{f})$ which appears in $\rterm(\p{m}, \p{g})$ for some other Fulton generator $\p{g}$ such that $\deg(\p{g}) = \deg(\p{f})$ and $\p{g} >_r \p{f}$. Write $\p{t} = \p{\tilde{t}} \cdot \overline{\p{t}}$ for $\tilde{\p{t}} \in \term(\p{f})$ and $\overline{\p{t}} \in \term(\overline{\p{f}})$ and let $\pset{M} = \{\mbox{Fulton generator } \tilde{\p{g}} ~|~ R(\tilde{\p{g}}) = R(\p{g}), C(\tilde{\p{g}}) \subseteq C(\p{m})\}$.  Then $\p{f}_{\overline{\p{t}}} \subseteq \rterm(\p{m}, \pset{M})$. 
\end{proposition}

\begin{proof}
  Let $R = R(\p{g})$ and $C = \bigcup_{\tilde{\p{g}} \in \pset{M}} C(\tilde{\p{g}})$. Write $\overline{\p{t}} = \overline{\p{t}}_1 \cdot \overline{\p{t}}_2$, where $\overline{\p{t}}_1$ contains the variables belonging to the sub-region of $\overline{R} \times \overline{C}$ (and thus to $(\overline{R} \times \overline{C}) \cap (R(\overline{\p{f}}) \times C(\overline{\p{f}}))$) and $\overline{\p{t}}_2$ contains those which do not. For each variable $x_{ij}$ in $\p{t} \in \rterm(\p{m}, \p{g})$, we know that if $i \in R$ then $j \in C(\p{g}) \subseteq C$; furthermore, for each $i \in R$, the corresponding variable $x_{ij}$ in $\p{t} \in \rterm(\p{m}, \p{f})$ cannot be in $\overline{\p{t}}_1$ and thus is either in $\tilde{\p{t}}$ or $\overline{\p{t}}_2$, and this implies that $R \subseteq R(\p{f}) \cup R(\overline{\p{t}}_2)$. For each $x_{ij}$ in $\overline{\p{t}}_2$, if $i \not \in R$ then $j \not \in \overline{C}$ and thus $j \in C$. By Corollary~\ref{cor:sameRow} we have $R \neq R(\p{f})$, and with $\p{g} >_r \p{f}$ we know that condition~(2) for the definition of $<_r$ holds. Then the structure of Fulton generators implies that $C(\p{f}) \subseteq C$.

  Consider a new minor $\tilde{\p{m}}$ with $R(\tilde{\p{m}}) = R(\p{f}) \cup R(\overline{\p{t}}_2)$ and $C(\tilde{\p{m}}) = C(\p{f}) \cup C(\overline{\p{t}}_2)$. Since $\overline{\p{t}}_2$ only contains variables in $R(\overline{\p{f}}) \times C(\overline{\p{f}})$, we know that $R(\p{f}) \cap R(\overline{\p{t}}_2)$ and $C(\p{f}) \cap C(\overline{\p{t}}_2)$ are both empty, and thus $\tilde{\p{m}}$ is indeed a minor. The analyses above imply the inclusions $C(\tilde{\p{m}}) \subseteq C$ and $R \subseteq R(\tilde{\p{m}})$. Now apply the generalized Laplace expansion to compute $\tilde{\p{m}}$ w.r.t. the rows $R$, we have $\tilde{\p{m}} = \sum_{R(\p{h}) = R} (-1)^{\delta(\p{h})}\p{h}\overline{\p{h}}$, where $\overline{\p{h}}$ is the complement of $\p{h}$ in $\tilde{\p{m}}$. Then
  $$\{\tilde{\p{t}} \cdot \overline{\p{t}}_2 ~|~ \tilde{\p{t}}\in \term(\p{f})\} \subseteq \term(\tilde{\p{m}}) \subseteq \cup_{R(\p{h}) = R}\term(\p{h}\overline{\p{h}}) = \rterm(\tilde{\p{m}}, \cup_{R(\p{h}) = R}\p{h}).$$
  Note that $\bigcup_{R(\p{h}) = R}\p{h} = \pset{M} \cap (R(\tilde{\p{m}}) \times C(\tilde{\p{m}}))$, with all the variables in $\overline{\p{t}}_1$ lying outside $R(\tilde{\p{m}}) \times C(\tilde{\p{m}})$ we have
  \begin{equation*}
      \begin{split}
   \{\tilde{\p{t}} \overline{\p{t}}_2 \cdot \overline{\p{t}}_1 ~|~ \tilde{\p{t}}\in \term(\p{f})\} & \subseteq \{\tilde{\p{t}} \cdot \overline{\p{t}}_1 ~|~ \tilde{\p{t}} \in \rterm(\tilde{\p{m}}, \cup_{R(\p{h}) = R}\p{h})\}\\
   & \subseteq \rterm(\p{m}, \cup_{R(\p{h}) = R}\p{h}) \subseteq \rterm(\p{m}, \pset{M}),       
      \end{split}
  \end{equation*}
  and this finishes the proof. 
\end{proof}

\begin{proposition}\label{prop:diffSize}
  Let $w \in S_n$ and $\p{m}$ be an elusive minor of $I_w$ reducible by a Fulton generator $\p{f}$. Suppose that there exists a term $\p{t} \in \rterm(\p{m}, \p{f})$ which appears in $\rterm(\p{m}, \p{g})$ for some other Fulton generator $\p{g}$ such that $\deg(\p{g}) > \deg(\p{f})$. Write $\p{t} = \p{\tilde{t}} \cdot \overline{\p{t}}$ for $\tilde{\p{t}} \in \term(\p{f})$ and $\overline{\p{t}} \in \term(\overline{\p{f}})$ and let $(p_{\p{f}}, q_{\p{f}})$ and $(p_{\p{g}}, q_{\p{g}})$ be the corresponding essential sets for $\p{f}$ and $\p{g}$ respectively. For the three cases of the relative positions of $(p_{\p{f}}, q_{\p{f}})$ and $(p_{\p{g}}, q_{\p{g}})$: (1) $p_{\p{f}} \leq p_{\p{g}}$ and $q_{\p{f}} \leq q_{\p{g}}$ (equality cannot hold for both), (2) $p_{\p{f}} > p_{\p{g}}$ and $q_{\p{f}} < q_{\p{g}}$, and (3) $p_{\p{f}} < p_{\p{g}}$ and $q_{\p{f}} > q_{\p{g}}$, let
  \begin{equation*}
    \pset{M} := \left\{
    \begin{tabular}[l]{ll}
       $\{\mbox{Fulton generator } \tilde{\p{g}} ~|~ R(\tilde{\p{g}}) = R(\p{g}), C(\tilde{\p{g}}) \subseteq C(\p{m})\}$ & Cases~(1-2), \\
      $\{\mbox{Fulton generator } \tilde{\p{g}} ~|~ C(\tilde{\p{g}}) = C(\p{g}), R(\tilde{\p{g}}) \subseteq R(\p{m})\}$ & Case~(3).
    \end{tabular}
    \right.
  \end{equation*}
Then $\p{f}_{\overline{\p{t}}} \subseteq \rterm(\p{m}, \pset{M})$.   
\end{proposition}

\begin{proof}
 The proof is similar to that for Proposition~\ref{prop:sameSize}, except that in case~(3) we apply generalized Laplace expansion w.r.t. the columns $C(\p{g})$. 
\end{proof}

Next we consider the relationship between the removed sets of Fulton generators and elusive minors contained in a minor. 

\begin{lemma}\label{lem:Fulton-elusive}
  Let $w \in S_n$ and $\p{m}$ be an elusive minor of $I_w$ reducible by a Fulton generator $\p{f}$. If $\p{f}$ is not an elusive minor, then there exists a set $\pset{M}$ of Fulton generators contained in $\p{f}$ such that all the minors in $\pset{M}$ are of the same degree smaller than $\deg(\p{f})$, $\p{f} \in \bases{\pset{M}}$, and $\rterm(\p{m}, \p{f}) \subseteq \rterm(\p{m}, \pset{M})$. 
\end{lemma}

\begin{proof}
  Since $\p{f}$ is not an elusive minor of $I_w$, we know that there exists $(p, q)\in \ess(w)$ such that $\p{f}$ attends the submatrix $X_{pq}$, which means that the intersection of $\p{f}$ and $X_{pq}$ contains more than $\rank(w_{pq}^T)+1$ full rows or columns of $\p{f}$. Then there exists a Fulton generator $\tilde{\p{g}}$ from $X_{pq}$ contained in this intersection. If the intersection contains full rows, let
  $$\pset{M} := \{\mbox{Fulton generator } \p{g} \mbox{ from } X_{pq}~|~R(\p{g}) =  R(\tilde{\p{g}}), \p{g} \mbox{ contained in } \p{f}\};$$
  otherwise let $\pset{M}$ be the set of those $\p{g}$ with $C(\p{g}) = C(\tilde{\p{g}})$ instead. In either case, the generalized Laplace expansion of $\p{f}$ implies $\p{f} = \sum_{\p{g} \in \pset{M}} (-1)^{\delta(\p{g})} \p{g}\overline{\p{g}}_{\p{f}}$, where $\overline{\p{g}}_{\p{f}}$ is the complement of $\p{g}$ in $\p{f}$, and thus $\p{f} \in \bases{\pset{M}}$. 

Now consider an arbitrary term $\p{t} \in \rterm(\p{m}, \p{f})$. Write $\p{t} = \tilde{\p{t}} \cdot \overline{\p{t}}$ with $\tilde{\p{t}} \in \term(\p{f})$ and $\overline{\p{t}} \in \term(\overline{\p{f}})$. With $\p{f}$ contained in $\p{m}$, $\p{g}$ is also contained in $\p{m}$ for each $\p{g} \in \pset{M}$. Note that in $\p{m}$, $\overline{\p{f}}$ is contained in the complement $\overline{\p{g}}_{\p{m}}$ of $\p{g}$ and the complement of $\overline{\p{f}}$ in $\overline{\p{g}}_{\p{m}}$ is $\overline{\p{g}}_{\p{f}}$. From the expansion above we also know that $\term(\p{f}) \subseteq \bigcup_{\p{g} \in \pset{M}}\term(\p{g}\overline{\p{g}}_{\p{f}})$. That is, for $\tilde{\p{t}}$ there exists a minor $\p{g} \in \pset{M}$ and $\tilde{\p{t}}_{\p{g}} \in \term(\p{g})$ and $\tilde{\p{t}}_{\overline{\p{g}}} \in \term(\overline{\p{g}})$ such that $\tilde{\p{t}} = \tilde{\p{t}}_{\p{g}} \tilde{\p{t}}_{\overline{\p{g}}}$. Fix this $\p{g}$ for $\tilde{\p{t}}$. Then for $\overline{\p{t}}$ we know that $\overline{\p{t}} \cdot \tilde{\p{t}}_{\overline{\p{g}}} \in \term(\overline{\p{g}}_{\p{m}})$, and thus $\p{t} = \tilde{\p{t}} \cdot \overline{\p{t}} = \tilde{\p{t}}_{\p{g}} \tilde{\p{t}}_{\overline{\p{g}}} \cdot \overline{\p{t}} \in \term(\p{g}\overline{\p{g}}_{\p{m}}) \subseteq \rterm(\p{m}, \pset{M})$. By the arbitrariness of $\p{t}$, we have $\rterm(\p{m}, \p{f}) \subseteq \rterm(\p{m}, \pset{M})$. 
\end{proof}

In particular, if we fix an (anti-)diagonal term order and let $\tilde{\p{t}} = \tilde{\p{t}}_{\p{g}} \cdot \tilde{\p{t}}_{\overline{\p{g}}}$ be the leading term of $\p{f}$ in the proof above, then $\tilde{\p{t}}_{\p{g}}$ is also the leading term of $\p{g}$ and thus $\lt(\p{f}) \in \lt(\pset{M})$. 

\begin{proposition}\label{prop:sameRemovedSet}
  For any elusive minor $\p{m}$ of $I_w$ for some $w \in S_n$, let $\pset{E}_{\p{m}}$ and $\pset{F}_{\p{m}}$ be the respective sets of elusive minors and Fulton generators strictly contained in $\p{m}$. Then $\rterm(\p{m}, \pset{M}_{\p{m}}) = \rterm(\p{m}, \pset{F}_{\p{m}})$ and $\lt(\pset{E}_{\p{m}}) = \lt(\pset{F}_{\p{m}})$ for any (anti-)diagonal term order. 
\end{proposition}

\begin{proof}
  On one hand, with $\pset{E}_{\p{m}} \subseteq \pset{F}_{\p{m}}$, clearly we have $\rterm(\p{m}, \pset{E}_{\p{m}}) \subseteq \rterm(\p{m}, \pset{F}_{\p{m}})$ and $\lt(\pset{E}_{\p{m}}) \subseteq \lt(\pset{F}_{\p{m}})$. On the other hand, for each $\p{f} \in \pset{F}_{\p{m}} \setminus \pset{E}_{\p{m}}$, by using Lemma~\ref{lem:Fulton-elusive} we have a set $\pset{M}_{\p{f}}$ of Fulton generators such that $\rterm(\p{m}, \p{f}) \subseteq \rterm(\p{m}, \pset{M}_{\p{f}})$ and $\lt(\p{f}) \in \lt(\pset{M}_{\p{f}})$. Write $\pset{M}_{\p{f}} = \pset{M}_{\p{f}}^e \cup \pset{M}_{\p{f}}^n$, where $\pset{M}_{\p{f}}^e$ contains all the elusive minors in $\pset{M}_{\p{f}}$ and $\pset{M}_{\p{f}}^n = \overline{\pset{M}_{\p{f}}^e}$. Then for each $\p{g} \in \pset{M}_{\p{f}}^n$, apply Lemma~\ref{lem:Fulton-elusive} again to have $\pset{M}_{\p{g}} = \pset{M}_{\p{g}}^e \cup \pset{M}_{\p{g}}^n$ and $\lt(\p{g}) \in \lt(\pset{M}_{\p{g}})$.

  Repeat this process until there is no non-elusive minor in the set after application of Lemma~\ref{lem:Fulton-elusive}: this always happens, for the degrees of the non-elusive minors decrease strictly in the process and all the 1-minors are elusive ones. Let $\pset{E}_{\p{f}}$ be the set of all the elusive minors appearing in the process, we have $\rterm(\p{m}, \p{f}) \subseteq \rterm(\p{m}, \pset{E}_{\p{f}}) \subseteq \rterm(\p{m}, \pset{E}_{\p{m}})$ and $\lt(\p{f}) \in \lt(\pset{E}_{\p{f}}) \subseteq \lt(\pset{E}_{\p{m}})$. The arbitrariness of $\p{f}$ implies $\rterm(\p{m}, \pset{F}_{\p{m}} \setminus \pset{E}_{\p{m}}) \subseteq \rterm(\p{m}, \pset{E}_{\p{m}})$ and $\lt(\pset{F}_{\p{m}} \setminus \pset{E}_{\p{m}}) \subseteq \lt(\pset{E}_{\p{m}})$: this ends the proof. 
\end{proof}

Now we are ready to prove the main theorem of this section. 

\begin{proof}[Proof of Theorem~\ref{thm:reductionM}]
 Let $\pset{F}_{\p{m}}$ be the set of Fulton generators contained in $\p{m}$. To prove equation~\eqref{eq:reductionM}, with $\rterm(\p{m},\pset{F}_{\p{m}}) = \rterm(\p{m},\pset{E}_{\p{m}})$ and $\lt(\pset{F}_{\p{m}}) = \lt(\pset{E}_{\p{m}})$ proved in Proposition~\ref{prop:sameRemovedSet}, it suffices to prove
 $$\p{m} \xrightarrow[*]{\pset{F}_{\p{m}}} \sum_{\p{t} \in \term(\p{m})\setminus \rterm(\p{m},\pset{F}_{\p{m}})}c_{\p{t}}\p{t}$$
 instead. Let $\pset{F}_{\p{m}} = \{\p{f}_1, \ldots, \p{f}_s\}$ such that $\p{f}_1 >_r \cdots >_r \p{f}_s$ and denote by $\pset{F}_i = \{\p{f}_1, \ldots, \p{f}_i\}$ for $i \in [s]$. Consider the following sequence of reduction of $\p{m}$ modulo the minors in $\pset{F}_{\p{m}}$: 
 $$   \p{m}\xrightarrow[*]{\p{f}_1} \p{m}_1  \xrightarrow[*]{\p{f}_2} \cdots \xrightarrow[*]{\p{f}_s} \p{m}_s.$$
Next we prove that $\p{m}_i = \sum_{\p{t} \in \term(\p{m})\setminus \rterm(\p{m},\pset{F}_{i})}c_{\p{t}}\p{t}$ and $\p{m}\xrightarrow[*]{\pset{F}_i} \p{m}_i$ (meaning $\p{m}_i$ is reduced modulo $\pset{F}_i$) for each $i \in [s]$ by induction on $i$. With $\pset{F}_s = \pset{F}_{\p{m}}$, this will finish the proof.

 Clearly these two statements hold for $i=1$ by Proposition~\ref{prop:reduction}, and assume that they hold for $i=k~(1\leq k < s)$. Next we prove the case when $i=k+1$. Consider the removed terms $\rterm(\p{m}, \p{f}_{k+1}) = \bigcup_{\overline{\p{t}} \in \term(\overline{\p{f}_{k+1}})} \{\tilde{\p{t}}\cdot\overline{\p{t}}~|~\tilde{\p{t}} \in \term(\p{f}_{k+1})\}$. For each $\overline{\p{t}} \in \term(\overline{\p{f}_{k+1}})$, we study the relationships between the $\overline{\p{t}}$-section $\p{f}_{k+1, \overline{\p{t}}}$ of $\p{f}_{k+1}$ in $\p{m}$ and $\rterm(\p{m}, \pset{F}_k)$. Since $\p{f}_1 >_r \cdots >_r \p{f}_k >_r \p{f}_{k+1}$, we know that $\deg(\p{f}_{k+1}) \leq \deg(\p{f}_i)$ for each $i=1, \ldots, k$. If $\p{f}_{k+1, \overline{\p{t}}} \cap \rterm(\p{m}, \pset{F}_k) \neq \emptyset$, then there exists $\p{g} \in \pset{F}_k$ such that $\p{f}_{k+1, \overline{\p{t}}} \cap \rterm(\p{m}, \p{g}) \neq \emptyset$. No matter $\deg(\p{g}) = \deg(\p{f}_{k+1})$ (by using Proposition~\ref{prop:sameSize}) or $\deg(\p{g}) > \deg(\p{f}_{k+1})$ (by Proposition~\ref{prop:diffSize}), we know that there exists a subset $\pset{M} \subseteq \pset{F}_k$ such that $\p{f}_{k+1, \overline{\p{t}}} \subseteq \rterm(\p{m}, \pset{M}) \subseteq \rterm(\p{m}, \pset{F}_k)$. From these arguments we know that $\rterm(\p{m}, \pset{F}_{k+1})$ equals the union of $\rterm(\p{m}, \pset{F}_k)$ and all $\p{f}_{k+1, \overline{\p{t}}}$ such that $\p{f}_{k+1, \overline{\p{t}}} \cap \rterm(\p{m}, \pset{F}_k) = \emptyset$. 

 Let $\p{t} = \lt(\p{f}_{k+1})$. In the case when $\p{f}_{k+1, \overline{\p{t}}} \cap \rterm(\p{m}, \pset{F}_k) = \emptyset$, by the induction assumption we know that $\p{m}_k = \sum_{\p{t} \in \term(\p{m})\setminus \rterm(\p{m},\pset{F}_{k})}c_{\p{t}}\p{t}$, and thus $\p{f}_{k+1, \overline{\p{t}}} \subseteq \term(\p{m}_k)$. Then we have $\p{m}_k \xrightarrow[\p{t}\overline{\p{t}}]{\p{f}_{k+1}} \p{m}_{k, \overline{\p{t}}}$, where $\p{m}_{k, \overline{\p{t}}}$ is obtained by removing the terms in $\p{f}_{k+1, \overline{\p{t}}}$ from $\p{m}_k$. In the case when $\p{f}_{k+1, \overline{\p{t}}} \cap \rterm(\p{m}, \pset{F}_k) \neq \emptyset$, with $\p{f}_{k+1, \overline{\p{t}}} \subseteq \rterm(\p{m}, \pset{F}_k)$ and $\p{m}_k = \sum_{\p{t} \in \term(\p{m})\setminus \rterm(\p{m},\pset{F}_{k})}c_{\p{t}}\p{t}$, we know that for any term $\tilde{\p{t}}\overline{\p{t}} \in \p{f}_{k+1, \overline{\p{t}}}$,  $\tilde{\p{t}}\overline{\p{t}} \not \in \p{m}_k$. In particular, $\p{t}\overline{\p{t}} \not \in \p{m}_k$, and thus $\p{m}_k \xrightarrow[\p{t}\overline{\p{t}}]{\p{f}_{k+1}} \p{m}_{k}$.

 Clearly different sections of $\p{f}_{k+1}$ in $\p{m}$ are disjoint. This means that reduction of $\p{m}_k$ modulo $\p{f}_{k+1}$ amounts to removing the terms in $\p{f}_{k+1, \overline{\p{t}}}$ for those $\overline{\p{t}}$ such that $\p{f}_{k+1, \overline{\p{t}}} \cap \rterm(\p{m}, \pset{F}_k) = \emptyset$, namely $\p{m}_k \xrightarrow[*]{\p{f}_{k+1}} \p{m}_{k+1} = \sum_{\p{t} \in \term(\p{m})\setminus \rterm(\p{m},\pset{F}_{k+1})}c_{\p{t}}\p{t}$. Note that in the process of term removal, terms not in $\term(\p{m}_k)$ will not reappear. Then by the induction assumption that $\p{m}  \xrightarrow[*]{\pset{F}_k} \p{m}_k$, we know that no term in $\p{m}_k$ is divisible by $\lt(\p{f}_i)$ for $i=1, \ldots, k$, and thus no term in $\p{m}_{k+1}$ is divisible by $\lt(\p{f}_i)$ for $i=1, \ldots, k+1$, namely $\p{m}  \xrightarrow[*]{\pset{F}_{k+1}} \p{m}_{k+1}$. 
\end{proof}

Clearly applying Theorem~\ref{thm:reductionM} to all elusive minors of a Schubert determinantal ideal for a non-vexillary permutation will furnish the reduced \grobner basis w.r.t. an anti-diagonal term order, and we formulate this result in the theorem below. Note that formula~\eqref{eq:reductionM} allows us to directly write down the result of reduction of an elusive minor modulo the set of elusive minors strictly contained in it without actually performing reduction between the polynomials expanded from the minors. 

\begin{theorem}\label{thm:reducedGB}
    Let $w \in S_n$ be a non-vexillary permutation, $\pset{E}$ be the set of all the elusive minors of $I_w$,  and $<$ be an anti-diagonal term order. For each $\p{m} \in \pset{E}$, write $\p{m} := \sum_{\p{t} \in \term(\p{m})} c_{\p{t}}\p{t}$ with $c_{\p{t}} = \pm 1$ and let $\tilde{\p{m}} = \sum_{\p{t} \in \term(\p{m})\setminus \rterm(\p{m},\pset{E}_{\p{m}})}c_{\p{t}}\p{t}$, where 
    $\pset{E}_{\p{m}}$ is the set of elusive minors strictly contained in $\p{m}$. Then $\bigcup_{\p{m} \in \pset{E}} \tilde{\p{m}}$ forms the reduced \grobner basis of $I_w$ w.r.t. $<$.
\end{theorem}

From the view point of computation, one can now represent all the minors with their row and column indices in the whole process to compute the reduced \grobner basis of a Schubert determinantal ideal: construction of Fulton generators from the essential set (\grobner bases), identification of elusive minors using the definition (minimal \grobner bases), and reduction among elusive minors using Theorem~\ref{thm:reductionM} (reduced \grobner bases). The former two steps are obviously combinatorial, and with Theorem~\ref{thm:reductionM} the originally algebraic last step also becomes combinatorial. Representation of minors with row and column indices means that the corresponding algorithm for computing the reduced \grobner bases of Schubert determinantal ideals, which can be easily formulated from the above-mentioned process, only needs to handle integers and thus are expected to be highly efficient.

In fact, when $n$ is large, the process of expanding all elusive minors of $w\in S_n$ and performing reduction among them can be costly. Take a non-vexillary $w = [1,9,4,2,7,6,3,5,10,8] \in S_{10}$ for example, there are 91 elusive minors in $I_w$ and reduction among them after expansion took more than two hours by using the default procedures in the Computer Algebra System {\sc Maple} on a personal laptop computer. We have implemented an algorithm named \texttt{RedGBSchubert} for computing reduced \grobner bases of Schubert determinantal ideals based on Theorem~\ref{thm:reducedGB}, included in the software package DetGB\footnote{Available at www.cmou.net/DetGB.html} we developed. Our implementation only takes approximately 720 seconds to compute the reduced \grobner basis of $I_w$. For further experimental comparisons, the readers are referred to \cite{MSZ24D}.

\section{W-characteristic sets of Schubert determinantal ideals}
\label{sec:w-char}

As shown in Section~\ref{sec:ts}, W-characteristic sets serve as a bridge to study multivariate polynomial ideals in terms of both \grobner bases and triangular sets, two fundamental tools in symbolic computation \cite{W2016o,WDM20d}. In this section we study properties of W-characteristic sets of Schubert determinantal ideals w.r.t. (anti-)diagonal lexicographic orders. We first prove the following equivalent conditions for a lexicographic term order to be (anti)-diagonal.

\begin{proposition}\label{prop:anti-var}
  Let $<_v$ be a variable order in $\kx$. Then the lexicographic term order induced by $<_v$ is anti-diagonal (resp. diagonal) if and only if the greatest variable in any square submatrix of $X$ is at its north-east or south-west corner (resp. north-west or south-east corner). 
\end{proposition}

\begin{proof}
We only prove the anti-diagonal case. 
  
  $(\Longleftarrow)$ Consider a square submatrix $M$ of size $r$ in $X$ and $\p{m} = \det(M)$ with $R(\p{m}) = \{i_1, \ldots, i_r\}$ and $C(\p{m}) = \{j_1, \ldots, j_r\}$ such that $i_1 < \cdots < i_r$ and $j_1 < \cdots < j_r$.  By the assumption we know that $\lv(\p{m})$ is either $x_{i_1,j_r}$ or $x_{i_r, j_1}$. Consider the case when $\lv(\p{m}) = x_{i_1,j_r}$. The Laplace expansion of $M$ w.r.t. its first row (otherwise in the case of $\lv(\p{m}) = x_{i_r,j_1}$, w.r.t. the last row) is $\p{m} = \sum_{k=1}^r c_{1k}x_{i_1,j_k}\p{m}_{1k}$, where $c_{1k} = \pm 1$ and $\p{m}_{1k}$ is the determinant of the submatrix $M_{1k}$ of $M$ by removing its $1$st row and $k$-th column. With the lexicographic term order, we have $\lt(\p{m}) = \lt(x_{i_1,j_r}\p{m}_{1r})$. Repeat the above process to $M_{1r}$ and its submatrices until the submatrix obtained by removing one row and column becomes of size 1. Then we will have $\lt(\p{m}) = \prod_{k=1}^r x_{i_k, j_{r-k+1}}$, and by the arbitrariness of $M$ the corresponding lexicographic term order is anti-diagonal. 

 $(\Longrightarrow)$ Suppose that there exists a square submatrix $M$ whose greatest variable $x_{ij}$ is not at its north-east or south-west corner. Then $x_{ij}$ and either of the north-west and south-east corners of $M$ determines a $2\times 2$ minor $\tilde{\p{m}}$ (depending on the position of $x_{ij}$ in $M$). With $\tilde{\p{m}}$ contained in $\p{m}$, $x_{ij}$ is still the greatest variable in $\tilde{\p{m}}$. Then $\lt(\tilde{\p{m}})$ contains $x_{ij}$ and it is not the product of all its anti-diagonal elements: a contradiction. 
\end{proof}

With Proposition~\ref{prop:anti-var}, we can assign the following 4 ``scanning'' variable orders to $X$ such that the lexicographic term orders induced by them are anti-diagonal: (1 and 2) the North-East corner variable $x_{1n}$ is the greatest and the next greatest variable is determined by scanning row by row to the West (or column by column to the South), and we call the corresponding anti-diagonal lexicographic term orders induced by them the NEW and NES ones; (3 and 4) similarly we have the SWE and SWN term orders which are also anti-diagonal. Four corresponding scanning variable orders to induce lexicographic diagonal term orders can be defined in the same way. 

In the following we prove the normality of the W-characteristic sets of Schubert determinantal ideals for vexillary permutations w.r.t. the NEW term order, and the same result holds to for the remaining 3 lexicographic anti-diagonal term orders defined above. 

\begin{theorem}\label{thm:normal}
  Let $w$ be a vexillary permutation. Then the W-characteristic set $\pset{C}$ of $I_w$ is normal for the NEW term order. 
\end{theorem}

\begin{proof}
  With the NEW term order being anti-diagonal, by Theorem~\ref{thm:schubert-vex-reduced} we know that all the elusive minors form the reduced \grobner basis of $I_w$. Since $w$ is vexillary, we can assign a total order to all the pairs in $\ess(w)$: for $(p, q), (\tilde{p}, \tilde{q}) \in \ess(w)$, $(p, q) < (\tilde{p}, \tilde{q})$ if $p < \tilde{p}$ or $p = \tilde{p}$ and $q > \tilde{q}$.

  Let $\p{m}$ be an elusive minor in $\pset{C}$ of size $r+1$, $\lv(\p{m}) = x_{ij}$, and $(p, q)$ be the greatest pair in $\ess(w)$ such that $\p{m} \subseteq X_{pq}$ and $\rank(w_{pq}^T) = r$ (note that for all $(p, \tilde{q})$ with $\tilde{q}\neq q$, we have $\rank(w_{p\tilde{q}}^T) \neq \rank(w_{pq}^T)$, and thus here we are picking the greatest $p$ with the suitable $q$). Consider the set
  $$ \pset{S} = \{(\tilde{p}, \tilde{q}) \in \ess(w): (\tilde{p}, \tilde{q}) > (p, q) \mbox{ and } \rank(w_{\tilde{p}\tilde{q}}^T) < r\}$$
  with the following 2 cases.

  (1) $\pset{S} = \emptyset$. Consider the minor $\p{m}' = (\{i\} \cup [p-r+1, p], [r] \cup \{j\})$. Then clearly $\p{m}'$ is of size $r+1$ and $\lv(\p{m}') = x_{ij}$. Next we prove that $\p{m}'$ is elusive and it is the smallest w.r.t. the NEW term order among all elusive minors of size $r+1$ and the leading variable $x_{ij}$.

  (1.1) For any $(\tilde{p}, \tilde{q}) \in \ess(w)$ and $(\tilde{p}, \tilde{q}) > (p, q)$, if $\rank(w_{\tilde{p}\tilde{q}}^T) \geq r$, then $\p{m}'$ does not attend $X_{\tilde{p}\tilde{q}}$. So we only need to study those with $\rank(w_{\tilde{p}\tilde{q}}^T)<r$. But by $\pset{S} = \emptyset$, we know that $\p{m}'$ does not attend any $X_{\tilde{p}\tilde{q}}$.

  For any $(\tilde{p}, \tilde{q}) \in \ess(w)$ and $(\tilde{p}, \tilde{q}) < (p, q)$, we know that $\tilde{p}<p$ and thus $\p{m}'$ intersects $X_{\tilde{p}\tilde{q}}$ with full rows. Since $R(\p{m}') = \{i\} \cup [p-r+1, p]$, we know that $\#(R(\p{m}') \cap [\tilde{p}])$ is already the least possible. Assume that there exists $(\tilde{p}, \tilde{q}) < (p, q)$ such that $\p{m}'$ attends $X_{\tilde{p}\tilde{q}}$, namely $\#(R(\p{m}') \cap [\tilde{p}]) \geq \rank(w_{\tilde{p}\tilde{q}}^T)+1$, then $\#(R(\p{m}) \cap [\tilde{p}]) \geq \#(R(\p{m}') \cap [\tilde{p}]) \geq \rank(w_{\tilde{p}\tilde{q}}^T)+1$: this means that $\p{m}$ attend $X_{\tilde{p}\tilde{q}}$ and it is a contradiction.

(1.2) With the definition of the NEW term order, We only need to prove that $R(\p{m}')$ is the largest possible and $C(\p{m}')$ is the smallest possible, and these are obvious with $(p, q)$ the greatest.

(2) $\pset{S} \neq \emptyset$. Write $\pset{S} = \{(p_1, q_1), \ldots, (p_t, q_t) \}$ with $(p_i, q_i) > (p_{i+1}, q_{i+1})$ for $i=1, \ldots, t-1$. Denote $r_i = \rank(w_{p_iq_i}^T)$, and let $s_1 = \min(r_1, \ldots, r_t)$ and $s_i = \min(r_i, \ldots, r_t) - \min(r_{i-1}, \ldots, r_t)$ for $i = 2, \ldots, t$. Now consider the minor $\p{m}'$ such that $R(\p{m}') = \{i\} \cup [p-r+1, p]$ and
\begin{equation}
  \label{eq:w-char-case2}
C(\p{m}') = [s_1] \cup [q_1+1, q_1+s_2] \cup \cdots \cup [q_{t-1}+1, q_{t-1}+s_t] \cup [q_t+1, q_t+ r-r_t] \cup \{j\}.  
\end{equation}
Then $\#C(\p{m}') = \sum_{i=1}^ts_i + r - r_t + 1 = r+1$ and $\lv(\p{m}') = x_{ij}$. The way how the row and column indices of $\p{m}'$ are chosen is illustrated in Figure~\ref{fig:normal} for an example (with the entries of $\p{m}'$ in pink). Next we prove that $\p{m}'$ is elusive and the smallest w.r.t. the NEW term order.

(2.1) For any $(\tilde{p}, \tilde{q}) \in \ess(w)$ such that  $(\tilde{p}, \tilde{q}) > (p, q)$ and $\rank(w_{\tilde{p}\tilde{q}}^T) < r$, we know that $(\tilde{p}, \tilde{q}) \in \pset{S}$. Let $(\tilde{p}, \tilde{q}) = (p_k, q_k)$ with $1\leq k \leq t$. Then $\#(C(\p{m}') \cap [q_k]) = \sum_{i=1}^ks_k = \min(r_k, \ldots, r_t) < r_k+1$, and thus $\p{m}'$ does not attend $X_{\tilde{p}\tilde{q}}$. For any $(\tilde{p}, \tilde{q}) \in \ess(w)$ and $(\tilde{p}, \tilde{q}) < (p, q)$, the proof is exactly the same as in (1.2).

(2.2) Obviously $R(\p{m}')$ is the largest possible. Next we prove that $C(\p{m}')$ is the smallest possible. For $i=1, \ldots, t-1$, assume that there exists an integer from $[s_i+1, q_i]$ which belongs to $C(\p{m}')$, then  $\#(C(\p{m}') \cap [q_i]) = \sum_{j=1}^is_j + 1 = \min(r_i, \ldots, r_t) +1$. So there exists $k \in [i,t]$ such that $r_k = \min(r_i, \ldots, r_t)$ and $\#(C(\p{m}') \cap [q_k]) \geq \#(C(\p{m}') \cap [q_i]) = r_k +1$: this means that $\p{m}'$ attends $x_{p_kq_k}$ and it is a contradiction. 

What we prove above about $\p{m}'$ fully characterizes the smallest polynomial among the reduced \grobner basis of $I_w$ of the same leading variable w.r.t. the NEW term order, and next we use this characterization to prove that $\pset{C}$ is normal. That is, for any polynomial $\p{m} \in \pset{C}$, its initial $\ini(\p{m})$ does not involve any $\lv(\pset{C} \setminus \{\p{m}\})$. Let $\lv(\p{m}) = x_{ij}$. Assume that there exists a polynomial $\tilde{\p{m}} \in \pset{C} \setminus \{\p{m}\}$ such that $\lv(\tilde{\p{m}})$ appears in $\ini(\p{m})$. We draw contradictions with $\tilde{\p{m}}$ by considering the two cases of $\pset{S}$ above.

(1) When $\pset{S} = \emptyset$: with the elusive minor $\p{m}$ of the form $(\{i\} \cup [p-r+1, p], [r] \cup \{j\})$ and $\lv(\tilde{\p{m}}) < \lv(\p{m})$, we have $C(\tilde{\p{m}}) \subseteq [r]$ and thus $\#C(\tilde{\p{m}}) \leq r$, meaning that $\tilde{\p{m}}$ is not a minor from $(p,q)$. But in this case there is no $(\tilde{p}, \tilde{q}) \in \ess(w)$ such that $(\tilde{p}, \tilde{q}) > (p, q)$ and $\rank(w_{\tilde{p}\tilde{q}}^T) < r$: a contradiction. 

(2) When $\pset{S} = \{(p_1, q_1), \ldots, (p_t, q_t)\} \neq \emptyset$: with $C(\p{m})$ in the form of equation~\eqref{eq:w-char-case2} and $\lv(\tilde{\p{m}}) < \lv(\p{m})$, we know that $\#C(\tilde{\p{m}}) \leq r$ (and thus $\tilde{\p{m}}$ is not a minor from $(p,q)$) and that for each $i=1, \ldots, t$, $\#(C(\p{m}) \cap [q_i]) \leq \sum_{j=1}^is_j= \min(r_i, \ldots, r_t) \leq r_i$. Let $\tilde{\p{m}}$ be a minor from $(p_k, q_k) \in \pset{S}$ for some $k \in [t]$ and $\lv(\tilde{\p{m}}) = x_{\tilde{i}\tilde{j}}$. If $k=1$, then $\tilde{j} \leq s_1 = \min(r_1, \ldots, r_t) \leq r_1$ by equation~\eqref{eq:w-char-case2}, and thus $\#C(\tilde{\p{m}}) \leq r_1$: a contradiction. If $k>1$, with $\tilde{\p{m}}$ being a polynomial in the W-characteristic set $\pset{C}$, $C(\tilde{\p{m}})$ is also of the form of equation~\eqref{eq:w-char-case2} (with $t$ replaced by $k-1$, $r$ by $r_k$, and $j$ by $\tilde{j}$, obviously). Then with $\#(C(\p{m}) \cap [q_k]) \leq r_k$ we have $\#(C(\tilde{\p{m}}) \cap [q_k]) \leq r_k$ too: another contradiction. This completes the whole proof. 
\end{proof}

  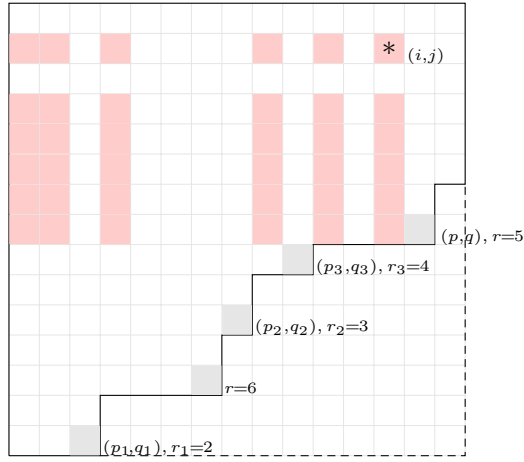
\begin{figure}[h]
    \centering

\begin{tikzpicture}
   \begin{scope}[scale =0.4]
        \draw[help lines,color=gray!20] (0,0) grid (15,15);
        \draw[black] (0,0)--(3,0)--(3,2)--(7,2)--(7,4)--(8,4)--(8,5)--(8,6)--
        (10,6)--(10,7)--(14,7)--(14,9)--(15,9)--(15,15)--(0,15)--(0,0);
         
        \node[font=\fontsize{6}{6}\selectfont,right,black] at (2.75,0.4){$(p_1\!,\!q_1),r_1\!\!=\!\!2$};
        \node[font=\fontsize{6}{6}\selectfont,right,black] at (7.75,4.4){$(p_2,\!q_2),r_2\!\!=\!\!3$};
        \node[font=\fontsize{6}{6}\selectfont,right,black] at (9.75,6.4){$(p_3,\!q_3),r_3\!\!=\!\!4$};
        \node[font=\fontsize{6}{6}\selectfont,right,black] at (6.75,2.4){$r\!\!=\!\!6$};
        \node[font=\fontsize{6}{6}\selectfont,right,black] at (13.85,7.4){$(p,\!q),r\!\!=\!\!5$};
        \node[font=\fontsize{6}{6}\selectfont,right,black] at (12.75,13.4){$(i,\!j)$};

        \filldraw[gray!20] (2.03,0.03) rectangle (2.97,0.97);
        \filldraw[gray!20] (6.03,2.03) rectangle (6.97,2.97);
        \filldraw[gray!20] (7.03,4.03) rectangle (7.97,4.97);
        \filldraw[gray!20] (9.03,6.03) rectangle (9.97,6.97);
        \filldraw[gray!20] (13.03,7.03) rectangle (13.97,7.97);

        \filldraw[red!20] (10.03,13.03) rectangle (10.97,13.97);
        \foreach \x in {7,...,11}{
         \filldraw[red!20] (10.03,\x+0.03) rectangle (10.97,\x+0.97);}
        
        \filldraw[red!20] (0.03,13.03) rectangle (0.97,13.97);
        \foreach \x in {7,...,11}{
            \filldraw[red!20] (0.03,\x+0.03) rectangle (0.97,\x+0.97);}

        \filldraw[red!20] (1.03,13.03) rectangle (1.97,13.97);
        \foreach \x in {7,...,11}{
            \filldraw[red!20] (1.03,\x+0.03) rectangle (1.97,\x+0.97);}

        \filldraw[red!20] (3.03,13.03) rectangle (3.97,13.97);
        \foreach \x in {7,...,11}{
            \filldraw[red!20] (3.03,\x+0.03) rectangle (3.97,\x+0.97);}
        
        \filldraw[red!20] (8.03,13.03) rectangle (8.97,13.97);
        \foreach \x in {7,...,11}{
            \filldraw[red!20] (8.03,\x+0.03) rectangle (8.97,\x+0.97);}
        
        \filldraw[red!20] (3.03,13.03) rectangle (3.97,13.97);
        \foreach \x in {7,...,11}{
            \filldraw[red!20] (3.03,\x+0.03) rectangle (3.97,\x+0.97);}
           
        \filldraw[red!20] (12.03,13.03) rectangle (12.97,13.97);
        \foreach \x in {7,...,11}{
            \filldraw[red!20] (12.03,\x+0.03) rectangle (12.97,\x+0.97);}
        
        \node[black] at (12.5,13.5){$\ast$};
        \draw[densely dashed,black](3,0)--(15,0)--(15,9);
   
    \end{scope}
\end{tikzpicture}
  \caption{Illustration for the construction of $\p{m}'$ in case (2) in the proof of Theorem~\ref{thm:normal}}\label{fig:normal}
  \end{figure}

The condition on $w$ to be vexillary is necessary as shown by the following example. Consider $w=1453276 \in S_7$ which is not vexillary with $\ess(w)=\{(3,3),(4,2),(6,6)\}$. Let $\pset{F}$ be the set of Fulton generators, namely all the 2-minors from $X_{33}$ and $X_{42}$ and the only 6-minor $G$ from $X_{66}$, and consider the NEW term order. Then one can easily check that all minors in $\pset{F}$ are elusive ones, and thus they form a minimal \grobner basis of $I_w$ by Theorem~\ref{thm:minimal}, but they are not the reduced one by Theorem~\ref{thm:schubert-reduced-vex}. After performing the reduction $G\xrightarrow[\ast]{\mathcal{F}\backslash\{G\}}\overline{G}$ to have a polynomial $\overline{G}$ in the reduced \grobner basis by using Theorem~\ref{thm:reductionM}, we find that $\lv(\overline{G}) = x_{16}$ and a term $x_{25}x_{32}x_{43}x_{51}x_{64}$ of $\ini(\overline{G})$ contains the leading 
variable $x_{32}$ of the elusive minor $(\{3,4\}, \{1,2\}) =x_{31}x_{42}-x_{32}x_{41}$. As a result, the W-characteristic set of $I_w$ is not normal.

Recall that the characteristic pair of a polynomial ideal consists of its lexicographic reduced \grobner basis and W-characteristic set. Next we show that the characteristic pair $(\pset{G}, \pset{C})$ of any Schubert determinantal ideal is strong, namely $\bases{\pset{G}} = \sat(\pset{C})$. 

\begin{proposition}\label{prop:strong}
  Let $\ideal{p}$ be a prime ideal in $\kx$ and $(\pset{G}, \pset{C})$ be its characteristic pair. Then $\bases{\pset{G}} = \sat(\pset{C})$.
\end{proposition}

\begin{proof}
  Write $\pset{C} = [C_1, \ldots, C_r]$. We first prove that for each $i=1, \ldots, r$, $\ini(C_i) \not \in \bases{\pset{G}}$. Otherwise, $\ini(C_i) \in \bases{\pset{G}}$ implies $\ini(C_i)  \xrightarrow[*]{\pset{G}} 0$, and thus there exists a term $\p{t} \in \term(\ini(C_i))$ and $G\in \pset{G}$ such that $\lt(G) | \p{t}$, and thus $\lt(G) | \p{t}\lv(C_i)$. Note that $\p{t} \lv(C_i) \in \term(C_i)$, and this contradicts that $\pset{G}$ is reduced.

  With $\bases{\pset{C}} \subseteq \bases{\pset{G}} = \ideal{p} \subseteq \sat(\pset{C})$ by \cite[Proposition~3.1]{W2016o}, it suffices to prove that for any $F \in \sat(\pset{C})$, we have $F \in \ideal{p}$. Since $F \in \sat(\pset{C})$, there exists an integer $k$ such that $FI^k \in \bases{\pset{C}} \subseteq \ideal{p}$, where $I = \prod_{i=1}^r \ini(C_i)$. Since $i=1, \ldots, r$, $\ini(C_i) \not \in \bases{\pset{G}}=\ideal{p}$, with $\ideal{p}$ prime and thus radical, we have $I \not \in \ideal{p}$, $I^k \not \in \ideal{p}$, and thus $F \in \ideal{p}$. 
\end{proof}

\begin{corollary}\label{cor:strong}
  The characteristic pair of any Schubert, one-sided, and two-sided ladder determinantal ideal is strong. 
\end{corollary}

\begin{proof}
  Straightforward because any Schubert, one-sided, and two-sided ladder determinantal ideal is prime by \cite{FUL92F}, Theorem~\ref{thm:one-sided-schubert}, and \cite[Theorem~1.18]{GOR07M} respectively. 
\end{proof}

\section{Concluding remarks}
\label{sec:con}

In this paper we propose a unified framework of blockwise determinantal ideals to study many existing kinds of determinantal ideals. Focusing on the reduction among \grobner bases of blockwise determinantal ideals, we establish several criteria to verify whether their \grobner bases are minimal or reduced ones w.r.t. (anti-)diagonal term orders. The reduced \grobner bases of Schubert determinantal ideals, the most representative blockwise determinantal ideals, are studied in detail. Fundamental properties of W-characteristic sets and characteristic pairs of Schubert determinantal ideals are also proven.

In existing study on determinantal ideals, specialization of some specific entries of the generic matrix is allowed. For example the CDG generators in \cite{HPW22G} are obtained by taking minors of the generic matrix after some specialization and they are also identified as \grobner bases of Schubert determinantal ideals. The current techniques developed in this paper for \grobner bases of blockwise determinantal ideals do not apply to generic matrices after specialization in general, and the extension to such matrices is under investigation. 

The results presented in Section~\ref{sec:w-char} first introduces the theory of triangular set to the study of Schubert determinantal ideals. In particular, the characteristic pair $(\pset{G}, \pset{C})$ of any Schubert determinantal ideal $I_w$ is proven to be strong, namely $I_w = \bases{\pset{G}} = \sat(\pset{C})$. With the rich interconnections between the initial ideal $\lt(I_w)$, the (bumpless) pipe dreams, Stanley-Reisner simplicial complex, and (double) Schubert polynomials, it would be very interesting to reveal the underlying relationships between W-characteristic sets and these combinatorial objects. 

\section*{Acknowledgements}
The authors would like to thank Weifeng Shang for his helpful discussions with us on lexicographic (anti-)diagonal term orders.

\bibliographystyle{amsplain}

\bibliography{GBforBlock}
\end{document}